%% file: main.tex
\documentclass[reqno,11pt]{amsart}

\input{preamble.tex}

\title{Density theorems for $\textnormal{GL}_n$ via Rankin--Selberg $L$-functions}
\author[J. D. Lichtman]{Jared Duker Lichtman}
\address{Department of Mathematics, Stanford University, Stanford, CA, USA}
\email{jared.d.lichtman@gmail.com}

\author[A. Pascadi]{Alexandru Pascadi}
\address{Mathematisches Institut, Endenicher Allee 60, 53115 Bonn, Germany}
\email{alexpascadi@gmail.com}

\begin{document}
\begin{abstract}
    We obtain density theorems for exceptional cuspidal automorphic representations of $\text{GL}_n$ over $\mathbb{Q}$ that fail the generalized Ramanujan conjecture at some place. We depart from approaches based on trace formulae, and instead rely on properties of families of Rankin--Selberg $L$-functions. This improves previous density results near the threshold of the pointwise bounds.
\end{abstract}

\maketitle

\vspace{-0.8cm}
{
\setlength{\parskip}{0em}
\setcounter{tocdepth}{1}
\tableofcontents
}
\vspace{-1.2cm}

\section{Introduction} \label{sec:intro}

Several results in analytic number theory \cite{topacogullari2018shifted,de2020niveau,drappeau2023one,wu2023fourth,lichtman2023primes} depend on the best progress towards the Ramanujan--Petersson conjecture and its archimedean counterpart, Selberg's eigenvalue conjecture. These concern the sizes of the Hecke and Laplacian eigenvalues of automorphic forms for congruence subgroups of $\SL_2(\Z)$, corresponding to the local parameters of $\GL_2$ automorphic representations. 
While the full conjectures seem to be out of the reach of current methods, it is desirable for applications to obtain partial results about the (conjecturally inexistent) \emph{exceptional forms}, which disobey the Ramanujan and Selberg bounds. In some cases, such substitutes can even match the best conditional results \cite{watt1995kloosterman,assing2021uniform,pascadi2026large}.

The $\GL_n$ setting \cite{luo1999generalized,blomer2011ramanujan,blomer2023density,assing2022density} has also attracted significant interest, in part because it leads to bounds for $\GL_2$ via symmetric power lifts \cite{luo1995selberg,kim2003functoriality}.
The \emph{Generalized Ramanujan Conjecture} (GRC), one of the key unsolved problems in number theory \cite{blomer2013role}, asserts that the local components of cuspidal automorphic representations of $\GL_n$ over a number field are tempered. For concreteness, let $\A_\Q$ denote the ring of adeles of $\Q$, $n \ge 2$, and 
$\mathfrak{F}_n$ be the family of all cuspidal automorphic representations of $\GL_n(\A_\Q)$, normalized so that their central characters are unitary and trivial on the diagonally-embedded positive reals (this ensures that no two $\pi, \pi' \in \mathfrak{F}_n$ are equivalent by an archimedean twist). Each
$\pi \in \mathfrak{F}_n$ has a generic, unitary, irreducible component at each place $v$, with associated Langlands parameters $\{\mu_{\pi,j}(v)\}_{j=1}^n$. At the unramified places $v$, the GRC predicts that
\[
    \Re\, \mu_{\pi,j}(v) = 0.
\]
What motivates the present work is a discrepancy between the two main ways to make progress towards the GRC. On the one hand, the best pointwise bounds (i.e., valid for any given representation $\pi$) rely on estimating sums of Dirichlet coefficients of Rankin--Selberg and symmetric power $L$-functions. Without further assumptions, the record at unramified places $v$ of $\Q$ is
\begin{equation} \label{eq:best-pointwise-bounds}
    |\Re\, \mu_{\pi,j}(v)| \le 
    \begin{cases} 
    \frac{1}{2} - \frac{1}{n^2+1}, & \text{if $n \ge 5$}, \\
    \frac{1}{2} - \frac{1}{\frac{n(n+1)}{2}+1}, & \text{if $n \in \{3, 4\}$}, \\
    \frac{7}{64}, & \text{if $n = 2$,}
    \end{cases} 
\end{equation}
due to Luo--Rudnick--Sarnak \cite{luo1995selberg,luo1999generalized} when $n \ge 5$ (see also \cite{serreletter,rudnick1996zeros}), and to Kim--Sarnak \cite[Appendix 2]{kim2003functoriality} otherwise. We also point the reader to \cite{muller2004absolute,blomer2011ramanujan} for the case of ramified places $v$, and to \cite{luo1999generalized,blomer2011ramanujan} for general number fields. 

On the other hand we have results holding on average, stating that there are few representations in a given family which fail the GRC by too much. The best such results have come from the spectral theory of automorphic forms \cite{iwaniec1997topics,humphries2018density,blomer2023density,assing2022density} and relied, through trace formulae \cite{deshouillers1982kloosterman,assing2022density}, on bounds for Kloosterman sums. For families corresponding to a congruence subgroup, Sarnak's density hypothesis \cite{sarnak1990diophantine,sarnak1991bounds} gives a prediction based on a linear interpolation between two extreme cases; one of these cases is the non-cuspidal trivial representation, so one should expect Sarnak's hypothesis to be suboptimal for cusp forms. Indeed, given a positive integer $q$ and $I \subset [0, \infty)$, let us denote by $\mF_I(q)$ the family of cuspidal automorphic representations of $\GL_n(\A_\Q)$ induced by Hecke--Maass forms which are invariant under $\Gamma_0(q) \subset \SL_n(\Z)$, and which have Laplacian eigenvalues $\lambda \in I$. The following result of Blomer \cite[Theorem 1]{blomer2023density} goes beyond Sarnak's density hypothesis for this family.

\begin{theorem}[Blomer \cite{blomer2023density}] \label{thm:blomer-density}
Let $v$ be a place of $\Q$, $q \neq v$ be a prime, and $I \subset [0, \infty)$ be a fixed compact interval. Then for any $\eps > 0$ and $\theta \in [0, \tfrac{1}{2})$, one has
\[
    \#\left\{\pi \in \mF_I(q) : \max_j |\Re\, \mu_{\pi,j}(v)| \ge \theta\right\}
    \ll_{n,v,I,\eps} 
    q^{n - 1 - 4\theta + \eps}.
\]
\end{theorem}
On $\GL_2$, this result is due to Iwaniec \cite{iwaniec1990small}; see also the results of Humphries \cite{humphries2018density}. We also mention Jana's density result for $\textnormal{PGL}_n(\Z)$ \cite[Theorem 3]{jana2021applications}, which matches Sarnak's density hypothesis.

A consequence of the two different methods leading to these pointwise and on-average results is that density theorems like \cref{thm:blomer-density} fail to see\footnote{Blomer \cite{blomer2023density} also remarks that it is not clear how to combine their spectral methods with the $L$-function techniques of Luo--Rudnick--Sarnak \cite{luo1999generalized}.} the sharp cutoffs in \cref{eq:best-pointwise-bounds}.
In particular, the upper bound in \cref{thm:blomer-density} is always $q^{n - O(1)}$, even though \cref{eq:best-pointwise-bounds} implies that the set is empty for $\theta > \tfrac{1}{2} - \tfrac{1}{n^2+1}$.
Given the great difficulty of improving the pointwise bounds, one may hope for a density bound that decays smoothly as $\theta$ approaches $\frac{1}{2}$. This is a relevant range for applications where the pointwise bound \cref{eq:best-pointwise-bounds} is currently used \cite{topacogullari2018shifted,de2020niveau,drappeau2023one,wu2023fourth,lichtman2023primes} (note that the range $\theta \approx \frac{1}{2}$ for $\GL_5$ impacts the range $\theta \approx \tfrac{1}{8}$ for $\GL_2$ via symmetric-fourth-power lifts, as in \cite[Appendix 2]{kim2003functoriality}).

In this paper, we prove density theorems for the local parameters of cuspidal automorphic representations of $\GL_n(\A_\Q)$, by studying sums over $\pi$ and $\pi'$ of the Dirichlet coefficients of Rankin--Selberg $L$-functions $L(s, \pi \times \tilde\pi')$; a key new ingredient is a positivity property for such sums, given in \cref{prop:rs-nonnegative-definite}. We achieve a smoother decay towards the pointwise thresholds in \cref{eq:best-pointwise-bounds}, and thus improve results like \cref{thm:blomer-density} for large values of $\theta$. Our main result is the following.

\begin{theorem}\label{thm:density}
Let $v$ be a place of $\Q$, $\mfC \ge 1$, and $\mF \subset \mathfrak{F}_n$ such that all $\pi \in \mF$ are unramified at $v$ and have total conductors $\mfC_\pi \le \mfC$ (as defined in \cref{eq:total-conductor})\footnote{This ensures that $\mF$ is finite \cite{brumley2006effective}, since total conductors from \cref{eq:total-conductor} are at least as large as analytic conductors.}. Then for any $\eps > 0$ and any $\theta \in (0, \tfrac{1}{2})$,
\[
    \#\left\{ \pi \in \mF : \max_j |\Re\, \mu_{\pi,j}(v)| \ge \theta \right\}
     \ll_{n,v,\theta,\eps} \mfC^{n\frac{1-2\theta}{2\theta} + \eps}.
\]
In fact, one can replace the base $\mfC^n$ in the right-hand side with $\max_{\pi, \pi' \in \mF} \mfC_{\pi \times \tilde\pi'}^{1/2}$.
\end{theorem}

\begin{remark}
When $v < \infty$, \cref{thm:density} holds, with the same proof, for the Iwaniec--Sarnak notion of analytic conductors \cite{iwaniec2000perspectives} instead of our choice in \cref{eq:total-conductor}. In this case, it is worthwhile to compare the bound in \cref{thm:density} to the total number of representations that it can consider. Indeed, let $\mathfrak{F}_n(\mfC)$ denote the \emph{universal family} of all $\pi \in \mathfrak{F}_n$ with analytic conductors up to $\mfC$. Then we expect \cite{brumley2024counting} that
\[
    |\mathfrak{F}_n(\mfC)| \asymp_n \mfC^{n+1}.
\]
In fact, Brumley--Mili\'cevi\'c \cite{brumley2024counting} proved this asymptotic when the family is restricted to the spherical representations, which in particular shows $\mfC \ll_n |\mathfrak{F}_n(\mfC)|^{1/(n+1)}$. So for $v < \infty$, \cref{thm:density} implies
\[
    \#\{\pi \in \mathfrak{F}_n(\mfC) \text{ unramified at } v : \max_j |\Re\, \mu_{\pi,j}(v)| \ge \theta\}
    \ll_{n,v,\theta,\eps}
    |\mathfrak{F}_n(\mfC)|^{\frac{n}{n+1} \frac{1-2\theta}{2\theta} + \eps}.
\]
Here the upper bound is nontrivial, and the exceptional set has density zero within $\mathfrak{F}_n(\mfC)$, provided
\begin{equation} \label{eq:nontrivial-range}
    \theta > \frac{1}{4} - \frac{1}{8n+4}.
\end{equation}
\end{remark}

\begin{remark}
The more explicit dependency on $\max_{\pi, \pi' \in \mF} \mfC_{\pi \times \tilde\pi'}^{1/2}$ in \cref{thm:density} is advantageous for `close-knit' families (in the sense of \cite[\S 1.5]{petrow2023fourth}), which exhibit Rankin--Selberg conductor dropping. One such family is $\mF := \{\pi \otimes \chi: \chi \in \Xi\}$, where $\pi \in \mathfrak{F}_n$ and $\Xi$ contains all the primitive even Dirichlet characters of a large prime conductor $q$. Then $|\mF| \asymp q$, and applying \cref{thm:density} with $\theta := \max_j |\Re\, \mu_{\pi,j}(v)|$ implies that 
\[
    q \ll_{n,v,\theta,\eps} \left(\mfC_\pi^{2n} q^{n^2}\right)^{\frac{1-2\theta}{4\theta} + \eps}
    \qquad 
    \stackrel{q \to \infty}{\Longrightarrow} 
    \qquad 
    \theta \le \frac{1}{2} - \frac{1}{n^2+2}.
\]
This falls slightly short of the Luo--Rudnick--Sarnak \cite{luo1995selberg,luo1999generalized} bound $\theta \le \tfrac{1}{2} - \tfrac{1}{n^2+1}$, because \cref{thm:density} does not exploit an additional square-root cancellation over Dirichlet characters (available for this particular family) through Deligne's bound for hyper-Kloosterman sums, as in \cite{duke1990estimates,luo1995selberg,kim2003functoriality}. Nevertheless, one can incorporate character twists into our arguments, to obtain a statement which recovers both \cref{thm:density} and the Luo--Rudnick--Sarnak bound, and which is in fact equivalent to the union of these two results. This is given in \cref{prop:density-amplif}, a particular case of which reads
\[
    q \ll_{n,v,\theta,\eps} \left(\mfC_\pi^{2n} q^{n^2-1}\right)^{\frac{1-2\theta}{4\theta} + \eps}
    \qquad 
    \stackrel{q \to \infty}{\Longrightarrow} 
    \qquad 
    \theta \le \frac{1}{2} - \frac{1}{n^2+1}.
\]
\end{remark}

For special families, one can exploit additional conductor dropping using character twists at ramified places. This leads to an enhanced result for the family considered by Blomer \cite{blomer2023density} in \cref{thm:blomer-density}, with an exponent that vanishes at the Luo--Rudnick--Sarnak threshold $\theta = \tfrac{1}{2} - \tfrac{1}{n^2+1}$. Combining this with the conclusion of \cref{thm:density} for the same family, we obtain the following.

\begin{theorem} \label{thm:density-blomer-family}
Let $v$ be a place of $\Q$, $q \neq v$ be a prime, and $I \subset [0, \infty)$ be a fixed compact interval. Then for any $\eps > 0$ and $\theta \in (0, \tfrac{1}{2})$, one has
\begin{equation} \label{eq:density-blomer-family-1}
    \#\left\{\pi \in \mF_I(q) : \max_j |\Re\, \mu_{\pi,j}(v)| \ge \theta\right\}
    \ll_{n,v,I,\eps} 
    q^{\min\left(\frac{n-1}{\theta}\left(\frac{1}{2} - \theta\right), \frac{n^2+1}{2\theta}(\frac{1}{2} - \frac{1}{n^2+1} - \theta)\right)+\eps}.
\end{equation}
\end{theorem}

In particular, \cref{eq:density-blomer-family-1} implies that for any $\delta > 0$, there can only be $O_{n,v,I,\delta}(q^{O(n^2\delta)})$ Hecke--Maass forms $\pi \in \mF_I(q)$ which get `$\delta$-close' to the Luo--Rudnick--Sarnak threshold, in the sense that $\max_j |\Re\, \mu_{\pi,j}(v)| \ge \tfrac{1}{2} - \tfrac{1}{n^2+1} - \delta$. Our proof of \cref{eq:density-blomer-family-1} (with the second choice in the minimum) crucially uses the fact that all representations in $\mF_I(q)$ have the same arithmetic conductor $q$.

Notably, since all $\pi \in \mF_I(q)$ have total conductors $\asymp_I q$, \cref{thm:density} directly implies the bound 
\begin{align} \label{eq:density-blomer-family-2}
    \sum_{\substack{q \le Q \\ v \nmid q}} \#\left\{\pi \in \mF_I(q) : \max_j |\Re\, \mu_{\pi,j}(v)| \ge \theta\right\}
    \ll_{n,v,I,\eps} 
    Q^{\frac{n}{\theta}(\frac{1}{2}-\theta)+\eps},
\end{align}
for $Q \ge 1$, summing over positive integers $q \le Q$ (in fact, \cref{thm:density} allows for arbitrary central characters, i.e., for including all nebentypen $\chi \pmod{d}$, $d \mid q$). As a quick comparison, summing \cref{thm:blomer-density} over the primes $q \le Q$ gives an upper bound of $Q^{n-4\theta+o(1)}$; when $n \ge 3$, our \cref{eq:density-blomer-family-2} improves this bound in the range
\[
    n\frac{1-2\theta}{2\theta} < n - 4\theta 
    \qquad\quad 
    \iff 
    \qquad\quad
    \theta > \frac{n - \sqrt{(n-2)n}}{4},
\]
which approaches the range $\theta \ge \frac{1}{4}$ as $n \to \infty$, much like \cref{eq:nontrivial-range}. The barrier at $\tfrac{1}{4} (n - \sqrt{(n-2)n})$ is always below the pointwise threshold from \cref{eq:best-pointwise-bounds}, so \cref{eq:density-blomer-family-2} gives a new result for $n \ge 3$. 

When $n = 2$, one should first take the symmetric fourth power lifts \cite{kim2003functoriality} of the $\GL_2$ representations considered, and then apply \cref{thm:density} for the resulting family of $\GL_5$ representations. Currently, this fails to beat the existing density theorems \cite{humphries2018density} for $\theta$ below the $\tfrac{7}{64}$ threshold of Kim--Sarnak, even when summing over all levels $q \le Q$ and all nebentypen $\chi \pmod{d}$, $d \mid q$. However, it is possible that \cref{thm:density} may be improved using similar ideas, and it would be interesting to obtain new results in the $\GL_2$ setting this way. For applications, it would be most relevant to obtain a nontrivial density theorem with only one form per level, as below; such results seem inaccessible to spectral methods, which rely on trace formulae for individual congruence subgroups.

\begin{problem} \label{prob:avg-level}
    Let $v$ be a place of $\Q$. For $q \in \Z_+$, let $\theta_q := \max_{\pi,j} |\Re\, \mu_{\pi,j}(v)|$, where $\pi$ ranges over cuspidal automorphic representations of $\GL_2(\A_\Q)$, generated by $\Gamma_0(q)$-invariant Maass forms. Show that, for all $Q \ge 1$ and some explicit $\eps, \delta > 0$,
    \[
        \# \left\{q \le Q : \theta_q > \frac{7}{64} - \delta \right\} \ll_v Q^{1-\eps}.
    \]
\end{problem}

This would mean beating the Kim--Sarnak bounds \cite[Appendix 2]{kim2003functoriality} for almost all levels $q$.

Indeed, such a density theorem with averaging over the levels could be combined with fixed-level \emph{large sieve inequalities} for exceptional Maass forms, which arise in the dispersion method \cite{deshouillers1982kloosterman,drappeau2017sums,pascadi2026large} (these are similar to density theorems, but incorporate additional information about the orthogonality of coefficients, which is useful in bounding multilinear forms of Kloosterman sums)\footnote{There exist large sieve inequalities for exceptional Maass forms which use averaging over levels (see \cite[Theorems 6, 7]{deshouillers1982kloosterman}, \cite{watt1995kloosterman}), but the same techniques do not seem to apply to proving density theorems with averaging over levels.}. This would automatically improve several results which currently rely on the pointwise bounds of Kim--Sarnak  \cite{topacogullari2018shifted,de2020niveau,drappeau2023one,lichtman2023primes}. We note that for the results at finite places $v$, the dependency of the implied constant on $v$ can be important in applications.


\begin{remark}
Building on the techniques in \cite{luo1999generalized,blomer2011ramanujan}, it should be possible to extend our results to arbitrary number fields. Using explicit descriptions of ramified local parameters of Rankin--Selberg $L$-functions as in \cite{rudnick1996zeros}, it might also be possible to prove corresponding results at ramified places $v$.
\end{remark}

\subsection{General notation.}
We use the standard asymptotic notation $\ll$, $\asymp$, $o(x)$, $O(x)$, $\Omega(x)$, $\Theta(x)$ from analytic number theory, indicating dependencies of the implied constants through subscripts (e.g., $O_\eps(x)$). In particular, statements like $f(x) \ll x^{o(1)} g(x)$ should be read as $\forall \eps > 0,\, f(x) \ll_\eps x^\eps g(x)$. We may also use the symbols $\approx$ and $\lesssim$ for informal or approximate identities and inequalities. We write $m \sim M$ as shorthand for $M < m \le 2M$, $\one_S$ for the truth value ($0$ or $1$) of a statement $S$, $\|u\|_q$ for the $\ell^q$ norm of a sequence $(u_m)$ (with $q \in [1, \infty]$), and $\Z_+ =\{n\in\Z : n\ge1\}$ the positive integers. Throughout this paper, we fix an integer $n \ge 2$ (and let all implicit constants depend on it). See \cref{sec:preliminaries} for notation specific to automorphic representations and $L$-functions.

\section{Outline} \label{sec:outline}

This section gives a brief and informal outline of our method, for readers who are familiar with the properties of automorphic and Rankin--Selberg $L$-functions (see \cref{sec:preliminaries} for more background).

Let $v$ be a fixed place of $\Q$. First, consider the problem of bounding the local parameters $\mu_{\pi,j} = \mu_{\pi,j}(v)$ of a given $\pi \in \mathfrak{F}_n$, unramified at $v$, as in \cref{eq:best-pointwise-bounds}. Recall that these parameters appear in the local factors $L_v(s, \pi)$ and $L_v(s, \pi \times \tilde \pi)$. 

The ``trivial'' bound $|\Re\, \mu_{\pi,j}| \le \frac{1}{2}$ follows from the fact that the local factors $L_v(s, \pi \times \tilde\pi)$ have no poles in $\Re\, s > 1$ \cite{jacquet1981euler}. But there is another simple, global argument for this bound (which is somewhat redundant in this setting, but generalizes well): for any $\ell \ge 1$, one has
\begin{equation} \label{eq:simple-argument-finite}
    \lambda_{\pi \times \tilde \pi}(\ell) \le
    \sum_{m \sim \ell/2} \lambda_{\pi \times \tilde \pi}(m)
    \ll 
    \ell^{1+o(1)},
\end{equation}
by the nonnegativity of the Rankin--Selberg coefficients $a_{\pi \times \tilde \pi}(m)$ \cite{rudnick1996zeros} and the absolute convergence of the Dirichlet series of $L(s, \pi \times \tilde\pi)$ in $\Re\, s > 1$. When $v = p < \infty$, taking $\ell = p^k$, the coefficient $a_{\pi \times \tilde \pi}(\ell)$ grows (at least on a subsequence of $k$'s) like $\ell^{2 \max_j |\Re\, \mu_{\pi,j}(p)|}$; so letting $k\to \infty$, this gives a contradiction unless $|\Re\, \mu_{\pi,j}(p)| \le \tfrac{1}{2}$. Similarly, when $v = \infty$, one can instead bound
\begin{equation} \label{eq:simple-argument-infinite}
    \ell^\beta \lambda_{\pi \times \tilde \pi}(1) \le
    \sum_{m \le \ell} \lambda_{\pi \times \tilde \pi}(m) \left(\frac{\ell}{m}\right)^\beta 
    \ll 
    \ell^{1+o(1)},
\end{equation}
for $\beta \in \R$ such that $L(\beta, \pi \times \tilde\pi) = 0$, by shifting contours to $\Re\, s = 1+\eps$ (indeed, the zero of $L(s, \pi \times \tilde\pi)$ at $s = \beta$ cancels the pole of a Mellin-transformed smooth majorant; see \cref{lem:rs-convexity}). Taking $\beta = \mu_{\pi,j} + \bar \mu_{\pi,j}$, which must be a zero of $L(s, \pi \times \tilde\pi)$ to cancel the corresponding pole of $L_\infty(s, \pi \times \tilde\pi)$ (see \cref{eq:local-factors-rs,eq:local-params-unram-rs}), and letting $\ell \to \infty$, one recovers the bound $|\Re\, \mu_{\pi,j}(\infty)| \le \tfrac{1}{2}$.

One can improve this to $|\Re\, \mu_{\pi,j}| \le \frac{1}{2} - \frac{1}{n^2+1}$, as anticipated in \cref{eq:best-pointwise-bounds}, by considering a family of twisted $L$-functions. Indeed, the classical method of Landau--Serre \cite{serreletter} essentially uses twists by archimedean characters to achieve this bound at the finite places of $\Q$. Luo--Rudnick--Sarnak \cite{luo1995selberg} used twists by Dirichlet characters to obtain results of the same strength at the infinite place, and their method extends to general number fields \cite{luo1999generalized}. More details on the role of character twists in such results are given in \cref{subsec:amplif-unram}. 

When $n \le 4$, one can work with the symmetric square $\Sym^2 \pi$ instead of $\pi \times \tilde\pi$, which improves the bound to $|\Re\, \mu_{\pi,j}| \le \frac{1}{2} - \frac{1}{n(n+1)/2 + 1}$, using related ideas of Duke--Iwaniec \cite{duke1990estimates}. When $n = 2$, one can also combine these results with symmetric power lifts \cite[Appendix 2]{kim2003functoriality}. This explains \cref{eq:best-pointwise-bounds}.

Our proof of \cref{thm:density} uses neither character twists nor symmetric squares. Rather, we insert averaging over representations $\pi, \pi' \in \mF$ into the simpler argument from \cref{eq:simple-argument-finite,eq:simple-argument-infinite}, working with the Rankin--Selberg $L$-functions $L(s, \pi \times \tilde\pi')$. The \emph{diagonal terms} with $\pi = \pi'$ can be treated as before, but constitute only a $|\mF|^{-1}$-fraction of the total sum. In the \emph{off-diagonal terms} with $\pi \neq \pi'$, we obtain savings from the fact that $L(s, \pi \times \tilde\pi')$ has no poles -- so essentially, from the orthogonality of the coefficients of $L(s, \pi)$ and $L(s, \pi')$. This is reminiscent of the proof of mean-value estimates (large sieve inequalities) in \cite{duke2000problem,thorner2021unconditional,humphries2024zeros}.

Notably, the pointwise argument in \cref{eq:simple-argument-finite,eq:simple-argument-infinite} depended on the fact that the coefficients $a_{\pi \times \tilde\pi}(m)$ are nonnegative\footnote{Even when using $\Sym^2 \pi$ instead of $\pi \times \tilde\pi$, one deduces the absolute convergence of the Dirichlet series of $L(s, \Sym^2 \pi)$ in $\Re\, s > 1$ from that of $L(s, \pi \times \tilde\pi)$.}. Our argument uses, as a key input, a generalization of this fact: the Rankin--Selberg coefficients $a_{\pi \times \tilde\pi'}(m)$ form a \emph{positive semi-definite} matrix in $\pi, \pi' \in \mF$. In other words, for any weights $w_\pi \in \C$ and any positive integer $m$, one has
\begin{equation} \label{eq:sketch-nonneg-def}
    \sum_{\pi, \pi' \in \mF} w_\pi \bar w_{\pi'}\, \lambda_{\pi \times \tilde\pi'}(m) \ge 0,
\end{equation}
as we show in \cref{prop:rs-nonnegative-definite}.
This allows one to bound
\[
    \left\vert \sum_{m \le M} u_m \sum_{\pi, \pi' \in \mF} w_\pi \bar w_{\pi'}\, \lambda_{\pi \times \tilde\pi'}(m)\right\vert
    \le 
    \|u\|_\infty \sum_{m = 1}^\infty \Phi\left(\frac{m}{M}\right) \sum_{\pi, \pi' \in \mF} w_\pi \bar w_{\pi'}\, \lambda_{\pi \times \tilde\pi'}(m),
\]
for any complex sequence $(u_m)$ and a suitable smooth majorant $\Phi$; the sums $\sum_m \Phi(\frac{m}{M})\, \lambda_{\pi \times \tilde\pi'}(m)$ in the right-hand side can then be expressed in terms of the $L$-functions $L(s, \pi \times \tilde\pi')$. We consider such trilinear sums over $m, \pi, \pi'$, with additional weights of $(M/m)^{\beta_\pi + \bar\beta_{\pi'}}$, in \cref{prop:rs-triple-sums}. We ultimately apply this for the sequences $u_m = \one_{m = \ell}$, respectively $u_m = \one_{m = 1}$ (with $M = \ell$).

This argument leads to better bounds for averages like $\frac{1}{|\mF|^2}\sum_{\pi,\pi' \in \mF} w_\pi \bar w_{\pi'}\, \lambda_{\pi \times \tilde\pi'}(\ell)$, which can be exploited to produce density theorems: rather than seeking a contradiction when $\max_j |\Re\, \mu_{\pi,j}|$ is too large, we seek an upper bound for $|\mF|$ in terms of the smallest local parameter $\min_{\pi \in \mF} \max_j |\Re\, \mu_{\pi,j}|$ and the largest total conductor $\max_{\pi \in \mF} \mfC_\pi$. The conductor aspect is crucial in such results (unlike in the pointwise bounds from \cite{luo1995selberg,rudnick1996zeros,kim2003functoriality}), so we need to make all dependencies on it explicit; the convexity bounds of Li \cite{li2010upper} are essential here. Another difficulty is that we cannot simply let various parameters tend to $\infty$ as in \cite{luo1995selberg,rudnick1996zeros,kim2003functoriality}; we will need to carefully optimize such parameters. Thus for instance, in our results at finite places, concluding the argument requires a more explicit lower bound in \cref{eq:sketch-nonneg-def} when $m = p^k$, as well as an application of Tur\'an's lower bounds for power sums.

In \cref{sec:preliminaries}, we describe the formal setup for our methods by gathering various preliminaries from the literature.
In \cref{sec:bounds-rankin-selberg}, we study sums over $\pi, \pi' \in \mF$ of Rankin--Selberg coefficients and establish the aforementioned \cref{prop:rs-nonnegative-definite,prop:rs-triple-sums}, which may be of independent interest.
In \cref{subsec:non-archimedean,subsec:archimedean}, we carry out the argument described above, with averaging over $\pi, \pi' \in \mF$, to prove density theorems for the local parameters of $\GL_n$ automorphic representations at $v = p < \infty$ and then $v = \infty$. In \cref{subsec:amplif-unram}, we incorporate twists by Dirichlet characters at an unramified place into our arguments, which allows us to recover the pointwise bounds of Luo--Rudnick--Sarnak \cite{luo1995selberg}. In \cref{subsec:amplif-ram}, we execute a similar argument using twists at a ramified place, which exploits additional conductor dropping for the special family $\mF = \mF_I(q)$, leading to the proof of \cref{thm:density-blomer-family}.

\section{Preliminaries} \label{sec:preliminaries}

\subsection{Mellin transforms and the Gamma function}
Given a bounded smooth function $\Phi$ on $(0, \infty)$ with Schwartz decay towards $\infty$, we define its Mellin transform as
\[
    \tilde\Phi(s) := \int_0^\infty x^{s-1}\, \Phi(x)\, dx,
\]
in $\Re\, s > 0$. This function decays rapidly in vertical strips and satisfies the Mellin inversion formula
\begin{equation} \label{eq:mellin-inversion}
    \Phi(x) = \frac{1}{2\pi i} \int_{(\sigma)} x^{-s}\, \tilde\Phi(s)\, ds,
\end{equation}
for any $\sigma > 0$, where the integral is over the vertical line at $\Re\, s = \sigma$. We recall that the Gamma function is defined as the Mellin transform of $e^{-x}$,
\[
    \Gamma(s) := \int_0^\infty x^{s-1}\, e^{-x}\, dx,
\]
in $\Re\, s > 0$, and by meromorphic continuation otherwise. It can be estimated by Stirling's formula,
\[
    \log \Gamma(s) = 
    \left(s - \frac{1}{2}\right) \log s - s + \frac{\log (2\pi)}{2} + O_\eps(|s|^{-1}),
\]
valid in $|\Arg\, s| < \pi - \eps$.
For any $\sigma, t \in \R$ and $C, c > 0$ with $|\sigma| \le C$ and $|t| \ge c$, it follows that
\[
    \Gamma(\sigma + it)
    \asymp_{c,C}
    |t|^{\sigma - \frac{1}{2}} e^{-\frac{\pi}{2} |t|}.
\]

Since $\Gamma$ has no zeros, and poles only at the nonpositive integers, this automatically improves to
\begin{equation} \label{eq:gamma-explicit}
\begin{cases}
    \Gamma(\sigma + it) \gg (1+|t|)^{\sigma - \frac{1}{2}} e^{-\frac{\pi}{2}|t|},
    &
    \text{for } \sigma \ll 1,
    \\
    \Gamma(\sigma + it) \ll (1+|t|)^{\sigma - \frac{1}{2}} e^{-\frac{\pi}{2}|t|},
    &
    \text{for } \sigma \ll 1,\ \min_{m \in \Z_{\le 0}} |\sigma + it - m| \gg 1. 
\end{cases}
\end{equation}
It will also be convenient to use the common notation
\[
    \Gamma_\R(s) := \pi^{-s/2}\, \Gamma(s/2),
\]
for factors that appear in the functional equations of $L$-functions. From \cref{eq:gamma-explicit}, it follows that for any complex numbers $s, z$ with $\Re\, s \ll 1$, $\Re\, z \ll 1$, and $\Re (1-s+z) \ge \eps > 0$, one has
\begin{equation} \label{eq:gamma-quotients}
    \frac{\Gamma_\R(1 - s + \bar z)}{\Gamma_\R(s + z)}
    \ll_\eps (1 + |\Im(s+z)|)^{\frac{1}{2} - \Re\, s}.
\end{equation}

\subsection{Automorphic $L$-functions} \label{subsec:automorphic}
We refer the reader to the books of Bump \cite{bump1998automorphic} and Goldfeld--Hundley \cite{goldfeld2011automorphic,goldfeld2011automorphicII} for an introduction to automorphic representations and their $L$-functions. Here we follow the setup in \cite{rudnick1996zeros,luo1995selberg,kim2003functoriality}, up to some changes in notation; see the work of Jacquet \cite{jacquet1979principal} and Godement--Jacquet \cite{godement1972zeta} for details.

We recall that we write $\mathfrak{F}_n$ for the family of cuspidal automorphic representations $\pi$ of $\GL_n(\A_\Q)$, normalized to have central characters $\omega_\pi$ on $\Q^\times \backslash \A_\Q^\times$ that are unitary \emph{and} trivial on the diagonally-embedded positive reals $(0, \infty) \subset \A_\Q^\times$. 
For later convenience, throughout this subsection, we let $\pi$ be any cuspidal automorphic representation of $\GL_n(\A_\Q)$ with unitary central character; thus $\pi$ could be any archimedean twist of a representation in $\mathfrak{F}_n$.

The representation $\pi$ has an irreducible unitary generic local component $\pi_v$ at each place $v$ (which could be $\infty$ or a prime $p$), finitely many of which may be ramified. At all places, ramified or not, the local factors $L_v(s, \pi)$ are of the form
\begin{equation} \label{eq:local-factors}
    L_v(s, \pi) =
    \begin{cases} 
    \prod_{j=1}^n \left(1 - \alpha_{\pi,j}(p)\, p^{-s} \right)^{-1},
    & v = p < \infty, \\
    \prod_{j=1}^n \Gamma_\R(s - \mu_{\pi,j}(\infty)),
    & v = \infty.
    \end{cases}
\end{equation}
To bring the finite and infinite places on an equal footing, for $v = p < \infty$ we denote 
\[
    p^{\mu_{\pi,j}(p)} := \alpha_{\pi,j}(p),
\] 
where $\Im\ \mu_{\pi,j}(p)$ is only determined modulo $\tfrac{2\pi}{\log p}$.
At the ramified primes we may have $\alpha_{\pi,j}(p) = 0$, case in which $\mu_{\pi,j}(p) = -\infty$.  The central character $\omega_\pi$ determines the sums $\sum_{j=1}^n \mu_{\pi,j}(v)$. 

At all places, the Satake/Langlands parameters $\mu_{\pi,j}$ satisfy \cite{jacquet1981euler}
\begin{equation} \label{eq:local-params-bound}
    \Re\, \mu_{\pi,j} \; < \; \frac{1}{2},
\end{equation}
which follows, remarkably, from purely local considerations \cite[(12), (13)]{sarnak2005notes}; 
see also \cite[(2.2), (2.5)]{rudnick1996zeros}. At the unramified places, one can insert absolute values in \cref{eq:local-params-bound}, and the unitarity assumption gives
\begin{equation} \label{eq:unitary-local-parameters}
    \left\{\mu_{\pi,j} : j \le n \right\}
    =
    \left\{-\bar \mu_{\pi,j} : j \le n \right\}.
\end{equation}
while the GRC asserts
\[
    \Re\, \mu_{\pi,j} = 0.
\]
For $v = \infty$ and $n = 2$, this corresponds to Selberg's eigenvalue conjecture \cite{selberg1965estimation,luo1995selberg}. 

The \emph{contragredient} $\tilde\pi$ is a unitary cuspidal automorphic representation of $\GL_n(\A_\Q)$, with local parameters given at every place by
\[
    \mu_{\tilde\pi,j} = \bar\mu_{\pi,j}.
\]
The global $L$-functions are then defined, in $\Re\, s > 1$, by
\[
    L(s, \pi) := \prod_{v = p < \infty} L_p(s, \pi) = \sum_{m = 1}^\infty \frac{\lambda_\pi(m)}{m^s},
    \qquad\qquad 
    \Lambda(s, \pi) := L_\infty(s, \pi)\, L(s, \pi),
\]
where $a_{\tilde\pi}(m) = \bar \lambda_\pi(m)$. The completed $L$-function $\Lambda(s, \pi)$ can be continued to an entire function \cite{godement1972zeta,jacquet1979principal}, which satisfies a functional equation of the shape
\begin{equation} \label{eq:functional-eqn}
    \Lambda(s, \pi) = 
    \epsilon(s, \pi)\,
    \Lambda(1-s, \tilde \pi).
\end{equation}
Here the epsilon factor is given by
\begin{equation} \label{eq:epsilon-factor}
     \epsilon(s, \pi) = \tau_\pi\, \mf_\pi^{-s},
\end{equation}
where $\mf_\pi$ is the \emph{(arithmetic) conductor} of $\pi$, and $\tau_\pi$ is a complex number with 
\begin{equation} \label{eq:gauss-sum}
    |\tau_\pi| = \sqrt{\mf_\pi}.
\end{equation} 
These satisfy $\mf_{\tilde \pi} = \mf_{\pi}$, $\tau_{\tilde \pi} = \bar \tau_{\pi}$, and one can factor
\[
    \epsilon(s, \pi) = \prod_v \epsilon_v(s, \pi), 
    \quad\quad 
    \epsilon_v(s, \pi) = \tau_{\pi}(v)\, \mf_{\pi \times \tilde\pi'}(v)^{-s},
    \quad\quad
    |\tau_{\pi}(v)| = \sqrt{\mf_{\pi}(v)},
\]
where $\mf_{\pi}(\infty) = 1$, $\mf_{\pi}(p)$ is a power of $p$, and $\epsilon_v(s, \pi) = 1$ when $v$ is unramified for $\pi$. The local factors $\epsilon_v(s, \pi)$ really depend on an implicit choice of an everywhere-normalized additive character $\psi = \prod_v \psi_v$ for $\A_\Q/\Q$, but we suppress this dependency in our notation.

Incorporating the contribution of the archimedean factor, we define\footnote{For technical reasons arising in the proof of \cref{lem:rs-convexity} with condition $(ii)$, our definition of $\mfC_\pi$ differs slightly from the standard \emph{analytic conductor} \cite{iwaniec2000perspectives,li2010upper,soundararajan2019weak}, which takes a product of factors $(1 + |\mu_{\pi,j}(\infty)|)$ rather than $n$ copies of the maximal factor. In particular, \cref{eq:total-conductor} gives a slightly larger value than the analytic conductor, so bounds like $F(\pi) \ll_\eps \mfC_\pi^\eps$ in \cite{li2010upper} remain true with our notation. However, for our results concerning the GRC at finite places $v$, one can use the analytic (rather than total) conductors, through condition $(i)$ of \cref{lem:rs-convexity}.} the \emph{total conductor} of $\pi$ by 
\begin{equation} \label{eq:total-conductor}
    \mfC_\pi := \mf_\pi \left(1 + \max_{1 \le j \le n} |\mu_{\pi,j}(\infty)|\right)^n.
\end{equation}

Then for $\Re\, s \in (-C,\frac{1}{2}-\eps]$ and $\eps,C > 0$, it follows from \cref{eq:gamma-quotients,eq:epsilon-factor,eq:local-params-bound,eq:gauss-sum,eq:functional-eqn,eq:local-factors,eq:total-conductor} that
\begin{equation} \label{eq:infty-quotients} 
    \frac{L(s, \pi)}{L(1-s, \tilde\pi)}
    =
    \epsilon(s, \pi) \frac{L_\infty(1-s, \tilde\pi)}{L_\infty(s, \pi)} \ll_{\eps} (1 + |s|)^{O(C)}\,
    \mfC_\pi^{\frac{1}{2} - \Re\, s}.
\end{equation}
One can \emph{twist} $\pi$ by any Hecke character of the idele class group $\Q^\times \backslash \A_\Q^\times$ (i.e., by an automorphic form on $\GL_1(\Q) \backslash \GL_1(\A_\Q)$), to obtain another automorphic representation $\pi \otimes \chi$ \cite[p.\,305]{bump1998automorphic}; this multiplies the underlying automorphic forms on $\GL_n(\Q) \backslash \GL_n(\A_\Q)$ by $\chi(\det(\cdot))$. If $\pi$ has central character $\omega_\pi$, then $\pi \otimes \chi$ has central character $\chi^n \omega_\pi$, so in particular $\pi \otimes \chi$ remains unitary if $\chi$ is unitary. As one would expect, the contragredient of $\pi \otimes \chi$ is $\tilde\pi \otimes \bar\chi$.

Following \cite[Appendix]{rudnick1996zeros}, when $\chi = |\cdot|^z$ for some $z \in \C$, we may also write
\[
    \pi[z] := \pi \otimes |\cdot|^z.
\] 
The archimedean twists by $|\cdot|^{it}$, for $t \in \R$, affect the $L$-function by $L(s, \pi[it]) = L(s + it, \pi)$; note that $\mathfrak{F}_n$ contains exactly one representative from each family of such archimedean twists. On the other hand, when $\chi$ is induced by a primitive even Dirichlet character of prime conductor $q \nmid \mf_\pi$ (as in \cite[Definition 2.1.7]{goldfeld2011automorphic}), we have $\pi \otimes \chi \in \mathfrak{F}_n$, and the twist affects the $L$-function by \cite{luo1995selberg}
\[
    L(s, \pi \otimes \chi) = \sum_{m = 1}^\infty \frac{\lambda_\pi(m) \chi(m)}{m^s},
    \qquad\qquad 
    L_\infty(s, \pi \otimes \chi) = L_\infty(s, \pi),
\]
and the epsilon factor by
\begin{equation} \label{eq:twisted-epsilon}
    \mf_{\pi \otimes \chi} = q^n\, \mf_\pi,
    \qquad\qquad 
    \tau_{\pi \otimes \chi} = 
    \tau_\chi^n\, \chi(\mf_\pi) \tau_\pi,
\end{equation}
where the classical Gauss sum of a primitive Dirichlet character is 
\begin{equation} \label{eq:gauss-sum-characters}
    \tau_\chi := \sum_{m = 1}^q \chi(m)\, e\left(\frac{m}{q}\right),
    \qquad\qquad 
    |\tau_\chi| = \sqrt{q}.
\end{equation}

We briefly mention how the local parameters of automorphic representations relate to the Laplacian and Hecke eigenvalues of automorphic forms. Let $q$ be a positive integer, and $\Gamma_0(q) \subset \SL_n(\Z)$ be the congruence subgroup of matrices with bottom row congruent to $(0, \ldots, 0, *)$ mod $q$. If $\pi$ is induced by a Maass form $F$ for $\Gamma_0(q)$ (which is a Hecke eigenform), we have the following:
\begin{itemize} 
    \item At $v = \infty$, the local parameters of $\pi$ correspond to the \emph{Laplacian eigenvalue} $\lambda_F(\infty)$ by \cite[p.\,49, pp.\,185-6]{terras1988harmonic}
    \[
    \lambda_F(\infty) = \frac{n^3-n}{24} - \frac{1}{2} \sum_{j=1}^n \mu_{\pi,j}(\infty)^2.
    \]
    Keeping $n$ fixed, this implies that one has $\lambda_F(\infty) \ll 1$ iff for each $j$, $\mu_{\pi,j}(\infty) \ll 1$ (since the real parts always satisfy $\Re\, \mu_{\pi,j}(\infty) \ll 1$).
    \item At the unramified primes $v = p < \infty$, the local parameters of $\pi$ correspond to the (appropriately normalized) $p^{th}$ \emph{Hecke eigenvalue} $\lambda_F(p)$ by
    \[
        \lambda_F(p) = \lambda_\pi(p) = \sum_{j=1}^n p^{\mu_{\pi,j}(p)}.
    \]
    \item The ramified primes $v = p < \infty$, and in fact the conductor $\mf_\pi$, divide the level $q$ \cite{jacquet1981conducteur,blomer2023density}. Thus the total conductor $\mfC_\pi$ encodes the growth of both the level and the Laplacian eigenvalue. If $\lambda_F(\infty) \ll 1$, it follows that $\mfC_\pi \ll q$.
    \item The central character $\omega_\pi$ is an adelization of the \emph{nebentypus} $\chi_F$ (indicating how $F$ transforms under the action of $\Gamma_0(q)$; this is taken to be trivial in \cref{thm:blomer-density}).
\end{itemize}
In particular, when $n = 2$, the Laplacian and Hecke eigenvalues and the nebentypus of $F$ uniquely determine the local parameters of $\pi$ at the unramified places (up to permutation).

\subsection{Rankin--Selberg $L$-functions} \label{subsec:rs}

Thoughout this subsection, we let $\pi, \pi'$ be cuspidal automorphic representations of $\GL_n(\A_\Q)$ with unitary central characters. If the Langlands conjectures are true, there should be a unitary automorphic representation $\pi \boxtimes \pi'$ of $\GL_{n^2}(\A_\Q)$. Unconditionally, we can still associate to the pair $(\pi, \pi')$ an $L$-function $L(s, \pi \times \pi')$, through the theory of Rankin--Selberg $L$-functions. This was developed by Jacquet, Piatetski-Shapiro and Shalika \cite{jacquet1983rankin}, Shahidi \cite{shahidi1981certain}, and Mœglin--Waldspurger \cite{moeglin1989spectre}; we also refer the reader to \cite{brumley2006effective,luo1995selberg,rudnick1996zeros,duke2000problem} for exposition, basic properties, and some applications.

Most properties of automorphic $L$-functions from the previous subsection have counterparts for Rankin--Selberg $L$-functions.
For later convenience, we will describe the $L$-function $L(s, \pi \times \tilde \pi')$ rather than $L(s, \pi \times \pi')$; the contragredient will help emphasize the `diagonal terms' with $\pi \simeq \pi'$. The finitely many ramified places of the pair $\pi \times \tilde \pi'$ are among the ramified places of $\pi$ or $\pi'$. At every place, ramified or not, the local factors have the form
\begin{equation} \label{eq:local-factors-rs}
    L_v(s, \pi \times \tilde \pi') = 
    \begin{cases} 
    \prod_{j = 1}^n \prod_{j' = 1}^n \left(1 - p^{-(s-\mu_{\pi \times \tilde \pi',j,j'}(p))} \right)^{-1},
    & v = p < \infty, \\
    \prod_{j = 1}^n \prod_{j' = 1}^n \Gamma_\R(s - \mu_{\pi \times \tilde \pi',j,j'}(\infty)),
    & v = \infty,
    \end{cases}
\end{equation}
where for $v = p < \infty$, one has $p^{\mu_{\pi \times \tilde \pi',j,j'}(p)} = \alpha_{\pi \times \tilde \pi',j,j'}(p)$, and $|\Im\, \mu_{\pi \times \tilde \pi',j,j'}(p)| \ll (\log p)^{-1}$. All local parameters at every place $v$ satisfy
\begin{equation} \label{eq:local-params-bound-rs}
    \Re\, \mu_{\pi \times \tilde\pi',j,j'} < 1.
\end{equation}
In fact, at the unramified places $v$ of both $\pi$ and $\pi'$, the local factors are explicitly given by
\begin{equation} \label{eq:local-params-unram-rs}
    \mu_{\pi \times \tilde \pi',j,j'}
    =
    \mu_{\pi,j} + \bar \mu_{\pi',j'},
\end{equation}
and the reader may compare \cref{eq:local-params-bound} with \cref{eq:local-params-bound-rs} in light of \cref{eq:local-params-unram-rs}. The global $L$-functions are then defined by
\[
\begin{aligned}
    L(s, \pi \times \tilde \pi') &:= \prod_{v = p < \infty} L_p(s, \pi \times \tilde \pi') = \sum_{m = 1}^\infty \frac{\lambda_{\pi \times \tilde \pi'}(m)}{m^s},
    \\
    \Lambda(s, \pi \times \tilde\pi') &:= L_\infty(s, \pi \times \tilde \pi')\, L(s, \pi \times \tilde \pi'),
\end{aligned}
\]
which converge absolutely in $\Re\, s > 1$; here the Dirichlet coefficients satisfy $a_{\tilde\pi \times \pi'}(m) = \bar \lambda_{\pi \times \tilde \pi'}(m)$.
The completed $L$-function can be continued to a meromorphic function on $\C$ such that \cite{jacquet1983rankin,moeglin1989spectre,brumley2006effective}
\[
    \Lambda(s, \pi \times \tilde \pi')\
    \begin{cases}
    \text{has (simple) poles only at $it$ and $1+it$},
    & \text{if } \exists t \in \R :  \pi[it] = \pi',
    \\
    \text{is entire}, 
    & \text{otherwise},
    \end{cases}
\]
and satisfies a functional equation (see also \cite[(2.3)]{luo1995selberg})
\begin{equation} \label{eq:functional-eqn-rs}
    \Lambda(s, \pi \times \tilde \pi') = 
    \epsilon(s, \pi \times \tilde \pi')\,
    \Lambda(1-s, \tilde \pi \times \pi').
\end{equation}
Here the epsilon factor is given by
\begin{equation} \label{eq:epsilon-factor-rs}
     \epsilon(s, \pi \times \tilde \pi') = \tau_{\pi \times \tilde \pi'}\, \mf_{\pi \times \tilde\pi'}^{-s},
\end{equation}
where $\mf_{\pi \times \tilde\pi'}$ is the (arithmetic) \emph{conductor} of $\pi \times \tilde\pi'$, and $\tau_{\pi \times \tilde\pi'}$ is a complex number with 
\begin{equation} \label{eq:gauss-sum-rs}
    |\tau_{\pi \times \tilde \pi'}| = \sqrt{\mf_{\pi \times \tilde\pi'}}.
\end{equation} 
As before, these satisfy $\mf_{\tilde \pi \times \pi'} = \mf_{\pi \times \tilde\pi'}$, $\tau_{\tilde \pi \times \pi'} = \bar \tau_{\pi \times \tilde \pi'}$, and one can factor
\[
    \epsilon(s, \pi \times \tilde\pi') = \prod_v \epsilon_v(s, \pi \times \tilde\pi'), 
    \quad\quad 
    \epsilon_v(s, \pi \times \tilde\pi') = \tau_{\pi \times \tilde\pi'}(v)\, \mf_{\pi \times \tilde\pi'}(v)^{-s},
    \quad\quad
    |\tau_{\pi \times \tilde\pi'}(v)| = \sqrt{\mf_{\pi \times \tilde\pi'}(v)},
\]
where $\mf_{\pi \times \tilde\pi'}(\infty) = 1$, and $\mf_{\pi \times \tilde\pi'}(p)$ is a power of $p$, and $\epsilon_v(s, \pi \times \tilde\pi') = 1$ when $v$ is unramified for $\pi \times \tilde\pi'$ (in particular, the primes dividing $\mf_{\pi \times \tilde\pi'}$ also divide $\mf_\pi\, \mf_{\pi'}$). Once again, the local factors $\epsilon_v(s, \pi \times \tilde\pi')$ implicitly depend on an everywhere-normalized additive character $\psi = \prod_v \psi_v$ for $\A_\Q/\Q$.

By Bushnell--Henniart \cite{bushnell1997upper} (see also \cite[p.\,22]{duke2000problem}), we have the bound
\begin{equation} \label{eq:conductor-bound-rs}
    \mf_{\pi \times \tilde \pi'} \le (\mf_\pi\, \mf_{\pi'})^n. 
\end{equation}
Like in the automorphic case, it is convenient to account for the contribution of the archimedean factor through the \emph{total conductor}
\begin{equation} \label{eq:total-conductor-rs}
    \mfC_{\pi \times \tilde \pi'} := \mf_{\pi \times \tilde \pi'} \left(1 + \max_{1 \le j,j'\le n}|\mu_{\pi \times \pi',j,j'}(\infty)|\right)^{n^2}.
\end{equation}
In particular, when $\pi$ and $\pi'$ are unramified at $\infty$, it is easily seen from \cref{eq:total-conductor,eq:total-conductor-rs,eq:conductor-bound-rs} that
\begin{equation} \label{eq:total-conductor-bound}
    \mfC_{\pi \times \tilde\pi'} \ll  \left(\mfC_\pi\, \mfC_{\pi'}\right)^n.
\end{equation}
In fact, the same holds true in the general case, which follows from (a simpler version of) the computations in \cite[Lemma 2.1]{soundararajan2019weak}; see also \cite[Lemma A.2]{humphries2019standard}.

In analogy with \cref{eq:infty-quotients}, for $\Re\, s \in (-C,-\eps]$ and $\eps, C > 0$, it follows from \cref{eq:gamma-quotients,eq:epsilon-factor-rs,eq:local-params-bound-rs,eq:gauss-sum-rs,eq:functional-eqn-rs,eq:local-factors-rs,eq:total-conductor-rs} that
\begin{equation} \label{eq:infty-quotients-rs} 
    \frac{L(s, \pi \times \tilde \pi')}{L(1-s, \tilde\pi \times \pi')}
    =
    \epsilon(s, \pi \times \tilde\pi') \frac{L_\infty(1-s, \tilde \pi \times \pi')}{L_\infty(s, \pi \times \tilde \pi')} \ll_{\eps} (1 + |s|)^{O(C)}\,
    \mfC_{\pi \times \tilde\pi'}^{\frac{1}{2} - \Re\, s}.
\end{equation}

We can also effectively \emph{twist} a Rankin--Selberg $L$-function by a Hecke character, by twisting one of the original two representations. For twists by $|\cdot|^{it}$, this affects the $L$-function by $L(s, \pi[it] \times \tilde\pi') = L(s + it, \pi \times \tilde\pi')$. The twists by Dirichlet characters\footnote{Note that we reserve the notation $\otimes$ for twists, and $\times$ for Rankin--Selberg $L$-functions.} work similarly, as shown below.

\begin{lemma}[Rankin--Selberg $L$-functions with unramified twists] \label{lem:twisted-epsilon-rs}
Let $\chi$ be a primitive, even Dirichlet character of prime conductor $q \nmid \mf(\pi) \mf(\pi')$. Then one has
\begin{equation} \label{eq:twisted-l-fn-rs}
    L(s, (\pi \otimes \chi) \times \tilde \pi') = \sum_{m = 1}^\infty \frac{\lambda_{\pi \times \tilde\pi'}(m) \chi(m)}{m^s},
    \qquad\quad 
    L_\infty(s, (\pi \otimes \chi) \times \tilde\pi') = L_\infty(s, \pi \times \tilde \pi').
\end{equation}
Moreover, there exist complex numbers $\eta_\pi(q), \bar\eta_{\pi'}(q)$ of absolute value $1$, depending only on $\pi$ and $q$ (not on $\chi$), such that
\begin{equation}\label{eq:twisted-epsilon-rs}
    \mf_{(\pi \otimes \chi) \times \tilde\pi'} = q^{n^2} \mf_{\pi \times \tilde\pi'},
    \qquad\qquad 
    \tau_{(\pi \otimes \chi) \times \tilde\pi'} = 
    \eta_\pi(q)\, \bar \eta_{\pi'}(q)\, \tau_\chi^{n^2} \chi(\mf_{\pi \times \tilde\pi'})\, \tau_{\pi \times \tilde\pi'},
\end{equation}
\end{lemma}
\begin{proof}
This follows with very minor modifications from the proof of \cite[Lemma 2.1]{luo1995selberg}, which considers the case $\pi = \pi'$. To establish \cref{eq:twisted-l-fn-rs,eq:twisted-epsilon-rs}, it suffices to compute the local factors $L_v(s, (\pi \otimes \chi) \times \tilde\pi')$, respectively $\epsilon_v(s, (\pi \otimes \chi) \times \tilde\pi')$. At $v = \infty$, even characters $\chi$ have a trivial component, so we have
\begin{equation} \label{eq:unram-infty-local}
    L_\infty(s, (\pi \otimes \chi) \times \tilde\pi') = L_\infty(s, \pi \times \tilde\pi'),
    \qquad 
    \qquad 
    \epsilon_\infty(s, (\pi \otimes \chi) \times \tilde\pi') = \epsilon_\infty(s, \pi \times \tilde\pi').
\end{equation}
If $v = p \neq q$, so $p$ is unramified for $\chi$, then we find just like in \cite[(2.14)]{luo1995selberg} that
\begin{equation} \label{eq:unram-prime-local}
\begin{aligned}
    L_p(s, (\pi \otimes \chi) \times \tilde\pi') &= 
    \prod_{j=1}^n \prod_{j'=1}^n \left(1 - \chi(p) p^{-(s-\mu_{\pi\times\tilde\pi',j,j'}(p))}\right)^{-1},
    \\
    \epsilon_p(s, (\pi \otimes \chi) \times \tilde\pi') &=
    \chi\left(\mf_{\pi \times \tilde \pi'}(p)\right)
    \epsilon_p(s, \pi \times \tilde\pi').
\end{aligned}
\end{equation}

If $v = q$, which is unramified for both $\pi$ and $\pi'$, then we find as in \cite[p.\,393--394]{luo1995selberg} that
\[
    L_q(s, (\pi \otimes \chi) \times \tilde\pi')
    =
    \prod_{j=1}^n \prod_{j'=1}^n \left(1 - \chi(q) q^{-(s-\mu_{\pi,j}(q) + \mu_{\pi',j'}(q))}\right)^{-1}
    =
    1
\]
and
\[
\begin{aligned}
    \epsilon_q(s, (\pi \otimes \chi) \times \tilde\pi') 
    &= 
    \prod_{j=1}^n \prod_{j' = 1}^n 
    \epsilon_q(s - \mu_{\pi,j}(q) + \mu_{\pi',j'}(q), \chi)
    \\
    &=
    \prod_{j=1}^n \prod_{j' = 1}^n 
    \tau_\chi\, q^{- s + \mu_{\pi,j}(q) - \mu_{\pi',j}(q)}
    \\
    &= 
    \eta_\pi(q)\, \bar\eta_{\pi'}(q)\, \tau_\chi^{n^2} \, q^{-n^2 s}, 
\end{aligned}
\]
where
\[
    \eta_\pi(q) := q^{\sum_{j=1}^n \mu_{\pi,j}(q)}
\]
is a complex number of absolute value $1$ by unitarity (the exponent is imaginary by \cref{eq:unitary-local-parameters}).
Taking a product over all places recovers \cref{eq:twisted-l-fn-rs,eq:twisted-epsilon-rs}.
\end{proof}

\begin{lemma}[Rankin--Selberg $L$-functions with ramified twists] \label{lem:twisted-epsilon-rs-ram}
Let $\chi$ be a primitive, even Dirichlet character of prime conductor $q$. Suppose $\pi, \pi'$ have arithmetic conductors $\mf_\pi = \mf_{\pi'} = q$ and trivial central characters. Then one has
\begin{equation} \label{eq:twisted-l-fn-rs-ram}
    L(s, (\pi \otimes \chi) \times \tilde \pi') = \sum_{m = 1}^\infty \frac{\lambda_{\pi \times \tilde\pi'}(m) \chi(m)}{m^s},
    \qquad\quad 
    L_\infty(s, (\pi \otimes \chi) \times \tilde\pi') = L_\infty(s, \pi \times \tilde \pi'),
\end{equation}
\begin{equation}\label{eq:twisted-epsilon-rs-ram}
    \mf_{(\pi \otimes \chi) \times \tilde\pi'} = q^{n^2},
    \qquad\qquad 
    \tau_{(\pi \otimes \chi) \times \tilde\pi'} =
    \tau_\chi^{n^2}\tau_{\pi \times \tilde\pi'}(\infty).
\end{equation}
\end{lemma}

\begin{proof}
As before, we can work locally. At $v = \infty$ and $v = p \neq q$, the same computations as in \cref{eq:unram-infty-local,eq:unram-prime-local} apply (and we now have $\mf_{\pi \times \tilde\pi'}(p) = 1$).

The only interesting place is $v = q$. We leave to \cref{subsec:twist-ram-chars} the (standard but quite technical) local computation that
\begin{equation} \label{eq:ram-twist-local}
    L_q(s, (\pi \otimes \chi) \times \tilde\pi') = 1, 
    \qquad
    \qquad 
    \epsilon_q(s, (\pi \otimes \chi) \times \tilde\pi') = \tau_\chi^{n^2} q^{-n^2s}.
\end{equation}
Then, taking a product over all places $v$ completes the proof.
\end{proof}

\section{Families of Rankin--Selberg \texorpdfstring{$L$}{L}-functions} \label{sec:bounds-rankin-selberg}

\subsection{Positive semi-definite coefficients of $L$-functions} \label{subsec:nonnegative-definite}

As discussed in \cref{sec:outline}, it is often a helpful property that the Dirichlet coefficients of a given $L$-function are nonnegative.
When dealing with families of $L$-functions, it is desirable to generalize this property to produce nonnegative averages of coefficients. A natural way to proceed is to consider positive semi-definite matrices.

Recall that a Hermitian matrix $M=(M_{i,j})\in \C^{N\times N}$ is {\it positive semi-definite} if and only if all eigenvalues of $M$ are nonnegative, equivalently $\vec{v}^{\,*}M \vec{v}\in \R_{\ge0}$ for all vectors $\vec{v}\in \C^N$.

\begin{definition}[Positive semi-definite families] \label{def:nonneg-def}
Let $\mI$ be a finite ordered set. For $i, j \in \mI$, let $L_{i,j}(s) = \sum_{m = 1}^\infty \lambda_{i,j}(m)\, m^{-s}$ be a formal Dirichlet series with complex coefficients.
We say that the family $(L_{i,j}(s))_{i,j \in \mI}$ is \emph{positive semi-definite} if and only if for any $m \ge 1$, the matrix $M \in \C^{\mI \times \mI}$ with entries
\[
    M_{i, j} := \lambda_{i, j}(m)
\]
is (Hermitian and) positive semi-definite. Note that this is independent of the ordering of the sequence $\mI$. When applied to complex $L$-functions, this definition refers to their Dirichlet expansions in $\Re\, s > \sigma$, for large enough $\sigma$.
\end{definition}

Recall that a matrix is positive semi-definite if and only if it may be written as a sum of rank-1 matrices of the form $\vec{w} \vec{w}^*$, e.g., via its eigendecomposition (here $\vec{w}^*$ denotes the conjugate transpose of a complex vector $\vec{w}$). It follows that $(L_{i,j}(s))_{i,j \in \mI}$ is a positive semi-definite family of $L$-functions if and only if it can be expressed as a sum of the shape
\begin{equation} \label{eq:positive-linear-combinations}
    L_{i,j}(s) = 
    \sum_{k=1}^\infty w_i(k) \bar w_j(k) m(k)^{-s},
    \qquad\qquad 
    w_i(k) \in \C, \quad m(k) \in \Z_+
\end{equation}
where $m(k)$ takes any given value $m$ only finitely-many times.

\begin{lemma}[$L$-function operations preserve positive semi-definiteness] \label{lem:properties-nonnegative-definite}
Let $\mI$ be a finite ordered set and $c \ge 0$; let $(L_{i,j}(s))_{i,j \in \mI}$ and $(L^{(k)}_{i,j}(s))_{i,j \in \mI}$ be positive semi-definite families of $L$-functions, for $k \ge 1$. Then the families
\[
    \left(c L_{i,j}(s)\right)_{i,j \in \mI},
    \qquad\quad
    \left( L^{(1)}_{i,j}(s) + L^{(2)}_{i,j}(s) \right)_{i,j \in \mI},
    \qquad\quad \text{and} \qquad\quad 
    \left( L^{(1)}_{i,j}(s) \cdot L^{(2)}_{i,j}(s) \right)_{i,j \in \mI}
\]
are positive semi-definite. Moreoever, if there exists a family of formal Dirichlet series $(L^{(\infty)}_{i,j})_{i,j \in \mI}$ such that $L^{(k)}_{i,j} \to L^{(\infty)}_{i,j}$ as $k\to\infty$ (in the sense of pointwise convergence of Dirichlet coefficients), then the limit family $(L^{(\infty)}_{i,j})_{i,j \in \mI}$ is also positive semi-definite. In particular, if well-defined, then the families
\[
    \left(\sum_{k = 1}^\infty L_{i,j}^k(s) \right)_{i,j \in \mI},
    \qquad\quad 
    \left(\prod_{k = 1}^\infty L_{i,j}^k(s) \right)_{i,j \in \mI},
    \qquad\quad\text{and} \qquad\quad
    \left(\exp(L_{i,j}(s)) \right)_{i,j \in \mI}.
\]
are positive semi-definite.
\end{lemma}

\begin{remark}
The last fact in \cref{lem:properties-nonnegative-definite} can be rephrased as follows: to show that the $L$-functions $L_{i,j}$ form a positive semi-definite family, it suffices to show the same for their formal logarithms $\log L_{i,j}$.
\end{remark}

\begin{proof}
The facts that a positive scaling of a positive semi-definite family, a sum of two positive semi-definite families, and the limit of a sequence of positive semi-definite families are positive semi-definite follow immediately from the corresponding matrix properties. For the product of two $L$-functions $L^{(k)}_{i,j} = \sum_{m \ge 1} \lambda_{i,j}^k(m)\, m^{-s}$, we note that the Dirichlet coefficients of $L^{(1)}_{i,j}(s) \cdot L^{(2)}_{i,j}(s)$ are given by
\[
    \lambda_{i,j}(m) = \sum_{d \mid m} \lambda_{i,j}^{(1)}(d)\, \lambda_{i,j}^{(2)}\left(\frac{m}{d}\right), 
\]
and the claim follows from Schur's product theorem that the Hadamard product of two positive semi-definite matrices is positive semi-definite. Equivalently, the fact that $( L^{(1)}_{i,j}(s) \cdot L^{(2)}_{i,j}(s))_{i,j \in \mI}$ is positive semi-definite is apparent from the characterization in \cref{eq:positive-linear-combinations}.

The last claim (about infinite sums, products, and exponentials) follows from the previous properties, noting that $\exp L(s) = \sum_{k \ge 0} \frac{1}{k!} L(s)^k$.
\end{proof}

We apply this notion to families of Rankin--Selberg $L$-functions. The following result is closely related to the computations of Brumley in \cite[Appendix]{soundararajan2019weak}.

\begin{proposition}[Rankin--Selberg $L$-functions are positive semi-definite] \label{prop:rs-nonnegative-definite}
For any finite set $\mF \subset \mathfrak{F}_n$, the family $(\log L(s, \pi \times \tilde{\pi}'))_{\pi, \pi' \in \mF}$
is positive semi-definite. In particular, so is $(L(s, \pi \times \tilde{\pi}'))_{\pi, \pi' \in \mF}$; thus for any $m \in \Z_+$ and $w_\pi \in \C$, one has
\[
    \sum_{\pi, \pi' \in \mF} w_\pi \bar w_{\pi'}\, \lambda_{\pi \times \tilde{\pi}'}(m)
    \ge 
    0.
\]
In fact, if $m = p^k$ is a prime power, where $k \ge 1$ and $p$ is an unramified prime for all $\pi \in \mF$, one has
\begin{equation} \label{eq:rs-explicit-lower-bound}
    \sum_{\pi, \pi' \in \mF} w_\pi \bar w_{\pi'}\, \lambda_{\pi \times \tilde{\pi}'}(p^k)
    \ge 
    \frac{1}{k}\left\vert \sum_{\pi \in \mF} w_\pi \sum_{j=1}^n \alpha_{\pi,j}(p)^k \right\vert^2.
\end{equation}
\end{proposition}

\begin{proof}
In what follows we work with the Dirichlet expansions of Rankin--Selberg $L$-functions in $\Re\, s > 1$; one may alternatively treat these as formal Dirichlet series. By \cref{lem:properties-nonnegative-definite}, positive semi-definiteness can be verified locally: i.e., it suffices to show that each local factor in $\log L(s, \pi \times \tilde{\pi}')=\sum_p \log L_p(s, \pi \times \tilde{\pi}')$ forms a positive semi-definite family in $\pi, \pi' \in \mF$.

If $p$ is unramified for both $\pi$ and $\pi'$, by \cite[(2.18)]{rudnick1996zeros} the (formal) logarithm of local factor has the form 
\begin{align*}
    \log L_p(s, \pi \times \tilde{\pi}') 
    &=
    \log \prod_{j,j' =1}^n \left(1 - \frac{\alpha_{\pi,j}(p)\, \bar \alpha_{\pi',j'}(p)}{p^s}\right)^{-1}
    \\
    &=
    -\sum_{j,j' =1}^n \log \left(1 - \frac{\alpha_{\pi,j}(p)\, \bar \alpha_{\pi',j'}(p)}{p^s}\right)
    \\
    &=
    \sum_{j,j' =1}^n \sum_{q = 1}^\infty \frac{\big(\alpha_{\pi,j}(p)\, \bar \alpha_{\pi',j'}(p)\big)^q}{q p^{qs}}
    =
    \sum_{q = 1}^\infty  \frac{\left(\sum_{j=1}^n \alpha_{\pi,j}(p)^q \right) \bar{\left(\sum_{j=1}^n \alpha_{\pi',j}(p)^q \right)}}{q p^{qs}},
\end{align*}
which is clearly positive semi-definite in $\pi, \pi'$. In fact, if $p$ is unramified for all $\pi \in \mF$, it follows that
\[
\begin{aligned}
    \sum_{\pi, \pi' \in \mF} w_\pi \bar w_{\pi'} \sum_{k=0}^\infty \frac{\lambda_{\pi \times \tilde{\pi}'}(p^k)}{p^{ks}}
    &=
    \sum_{\pi, \pi' \in \mF} w_\pi \bar w_{\pi'} \exp \left(\sum_{q = 1}^\infty  \frac{\left(\sum_{j=1}^n \alpha_{\pi,j}(p)^q \right) \bar{\left(\sum_{j=1}^n \alpha_{\pi',j}(p)^q \right)}}{q p^{qs}}\right)
    \\
    &= 
    \sum_{\ell = 1}^\infty \frac{1}{\ell!} 
    \sum_{\pi, \pi' \in \mF} w_\pi \bar w_{\pi'} \left(\sum_{q = 1}^\infty  \frac{\left(\sum_{j=1}^n \alpha_{\pi,j}(p)^q \right) \bar{\left(\sum_{j=1}^n \alpha_{\pi',j}(p)^q \right)}}{q p^{qs}}\right)^\ell.
\end{aligned}
\]
Each of the inner Dirichlet series is positive semi-definite in $\pi, \pi'$ in the sense of \cref{def:nonneg-def}. Thus for $k \ge 1$, identifying coefficients of $p^{-ks}$ and dropping all terms except for $\ell = 1$ by nonnegativity, we obtain
\[
\begin{aligned}
    \sum_{\pi, \pi' \in \mF} w_\pi \bar w_{\pi'}\, \lambda_{\pi \times \tilde{\pi}'}(p^k)
    &\ge 
    \sum_{\pi, \pi' \in \mF} w_\pi \bar w_{\pi'}  \frac{\left(\sum_{j=1}^n \alpha_{\pi,j}(p)^k \right) \bar{\left(\sum_{j=1}^n \alpha_{\pi',j}(p)^k \right)}}{k}
    \\
    & =
    \frac{1}{k} \left\vert \sum_{\pi \in \mF} w_\pi \sum_{j=1}^n \alpha_{\pi,j}(p)^k \right\vert^2.
\end{aligned}
\]
This proves \cref{eq:rs-explicit-lower-bound}. It remains to show that $\log L(s, \pi \times \tilde\pi')$ forms a positive semi-definite family in $\pi, \pi'$ when $p$ is an arbitrary prime, which may be ramified for some $\pi \in \mF$. This follows almost immediately from Brumley's local computations \cite[Formula (A.8)]{soundararajan2019weak}, after explicitating $J_a, z_j$ and $K_b, z_k'$ as functions of $\pi$ and $\pi'$; but for completeness, we include a proof in \cref{subsec:pos-def-ram-primes}.
\end{proof}

\begin{remark}
Taking $|\mF| = 1$ in \cref{prop:rs-nonnegative-definite} recovers the fact that the `diagonal' Rankin--Selberg $L$-functions $L(s, \pi \times \tilde\pi)$ have nonnegative Dirichlet coefficients \cite{rudnick1996zeros}, i.e.\ for all $m \ge 1$,
\[
    \lambda_{\pi \times \tilde\pi}(m) \ge 0.
\]
Taking $|\mF| = 2$, say $\mF = \{\pi, \pi'\}$, shows that the matrix 
\[
\begin{pmatrix}
    \lambda_{\pi \times \tilde \pi}(m) & \lambda_{\pi \times \tilde \pi'}(m) \\
    \lambda_{\pi' \times \tilde \pi}(m) & \lambda_{\pi' \times \tilde \pi'}(m)
\end{pmatrix}
\]
has nonnegative eigenvalues (in particular, nonnegative determinant), so
\[
    |a_{\pi \times \tilde \pi'}(m)| \le 
    \sqrt{\lambda_{\pi \times \tilde \pi}(m)\, \lambda_{\pi' \times \tilde \pi'}(m)}.
\]
Applying the same argument for the Dirichlet coefficients of $\log L(s, \pi \times \tilde\pi')$ recovers \cite[Lemma 2.2, first part]{soundararajan2019weak}. But in this paper, we will only use \cref{prop:rs-nonnegative-definite} for large families $\mF$ of representations.
\end{remark}

\begin{remark}
Writing $\log L(s, \pi \times \tilde\pi') = \sum_{m = 1}^\infty \frac{b_{\pi \times \tilde\pi'}}{m^s}$ in $\Re\, s > 1$, the positive semi-definiteness property in \cref{prop:rs-nonnegative-definite} states that $\sum_{\pi, \pi' \in \mF} w_\pi \bar w_{\pi'}\, b_{\pi \times \tilde\pi'} \ge 0$, for any complex weights $(w_\pi)$. If one was only interested in the case $w_\pi \equiv 1$, this would follow by considering the isobaric sum of all $\pi \in \mF$ (and the Rankin--Selberg convolution with its contragredient). Morally, our situation corresponds to a ``weighted isobaric sum'' with complex weights.
\end{remark}

\subsection{Triple sums of Rankin--Selberg coefficients} \label{subsec:sums-rs}
In the previous subsection, we looked at sums over $\pi$ and $\pi'$ of the Rankin--Selberg coefficients $a_{\pi \times \tilde\pi'}(m)$. Here we insert an additional sum over $m$, seeking upper bounds; we begin with the following lemma. We recall that $n$ is fixed.

\begin{lemma}[Smooth sums of Rankin--Selberg coefficients] \label{lem:rs-convexity}
Let $\pi, \pi' \in \mathfrak{F}_n$. Let $\beta \in \C$ and $\Phi(x) := x^{\beta} e^{-x^2}$ such that\footnote{We use this explicit choice of $\Phi$ to have fine control over its Mellin transform.} one of the following is true:
\begin{itemize}
    \item[$(i)$.] $0 < \Re\, \beta \asymp 1$ and $\Im\, \beta \ll 1$, or
    \item[$(ii)$.] $\pi, \pi'$ are unramified at $\infty$, and $-\beta = \mu_{\pi \times \tilde\pi',j,j'}(\infty)$ for some $1 \le j, j' \le n$.
\end{itemize}
Then for any $M \gg 1$, one has
\begin{equation} \label{lem:smooth-sum-bound}
    \sum_{m = 1}^\infty \Phi\left(\frac{m}{M}\right) \lambda_{\pi \times \tilde{\pi}'}(m)
    \ll
    (M \mfC_{\pi \times \tilde\pi'})^{o(1)} \cdot
    \begin{cases} 
    M,
    &{\rm if}\;
    \pi = \pi',
    \\
    \min(M, \sqrt{\mfC_{\pi \times \tilde\pi'}}),
    &{\rm otherwise}.
    \end{cases}
\end{equation}
\end{lemma}

\begin{remark}
Condition $(i)$ in \cref{lem:rs-convexity} will ultimately be relevant for our estimates at the finite places, while condition $(ii)$ will be relevant for the infinite place. Intuitively, ignoring the smooth weights, there are two relevant operations that one can apply to a sum of coefficients of an $L$-function, with Dirichlet series $L(s) = \sum_{m = 1}^\infty a(m) m^{-s}$ in $\Re\, s > 1$ and conductor $\mfC$. One is to bound it in terms of values of $L(s)$ near the $1$-line, which gives a bound of the shape (informally)
\[
    \sum_{m \sim M} a(m) \lesssim M^{1+o(1)}.
\]
The other is to pass to a dual sum of length $\mfC/M$ via an approximate functional equation, giving
\begin{equation} \label{eq:approx-func-eq-sketch}
    \sum_{m \sim M} a(m) \approx
    \mfC^{1/2} \sum_{m \sim \mfC/M} \frac{\tilde a(m)}{m}
    +
    \textnormal{Residues}.
\end{equation}
In the absence of poles, combining the two bounds above gives $\sum_{m \sim M} a(m) \lesssim \mfC^{1/2} \frac{\mfC}{M} (\frac{\mfC}{M})^{-1} = \mfC^{1/2}$.
\end{remark}

\begin{proof}[Proof of \cref{lem:rs-convexity}]
We first note that condition $(ii)$ implies $|\Re\, \beta| \le 1$; so under either of conditions $(i)$ and $(ii)$, we have $-1 \le \Re\, \beta \ll 1$ and
\begin{equation} \label{eq:beta-imaginary-condition} 
    \Im\, \beta \ll 1 + \max_{1 \le j,j'\le n} |\mu_{\pi \times \tilde\pi',j,j'}(\infty)|.
\end{equation} 
Let $\eps \in (0, \tfrac{1}{2})$; if $(i)$ holds, we also take $\eps < \tfrac{1}{2} \Re\, \beta$, so that in particular $|\Re\, \beta - \eps| \ge \eps/2$. By replacing $\pi$ with $\pi[- i u] = \pi \otimes |\cdot|^{-iu}$ and $\beta$ with $\beta - iu$ where $u = \Im\, \beta$, we may also assume that
\[
    \Im\, \beta = 0.
\]
Indeed, these substitutions make no impact on conditions $(i)$ and $(ii)$, and affect the left-hand side of \cref{lem:smooth-sum-bound} only by a factor of $M^{i\,\Im\,\beta}$, and the right hand side of \cref{lem:smooth-sum-bound} only by a constant. For the last claim, note that \cref{eq:beta-imaginary-condition} and our definition of total conductors from \cref{eq:total-conductor-rs} imply
\[
    \mfC_{\pi[-i\, \Im\, \beta] \times \tilde\pi'} \ll \mfC_{\pi \times \tilde\pi}.
\]
We now use Mellin inversion, as in \cref{eq:mellin-inversion}, to expand
\begin{equation} \label{eq:mellin-rs}
    \sum_{m=1}^\infty \Phi\left(\frac{m}{M}\right) \lambda_{\pi \times \tilde{\pi}'}(m)
    =
    \frac{1}{2\pi i} \int_{(2)} M^s L(s, \pi \times \tilde \pi')\, \tilde\Phi(s)\, ds.
\end{equation}
Here we can explicitly compute (in $\Re\, (s+\beta) > 0$, and by meromorphic continuation elsewhere)
\[
\begin{aligned}
    \tilde \Phi(s) 
    &= \int_0^\infty x^{s-1+\beta} e^{-x^2} dx
    \\
    &= \frac{1}{2}\int_0^\infty y^{(s-1+\beta)/2} e^{-y} y^{-1/2} dy
    =
    \frac{1}{2}\Gamma\left(\frac{s + \beta}{2}\right).
\end{aligned}
\] 
In particular, \cref{eq:gamma-explicit} and $\Re\, \beta \ll 1, \Im\, \beta = 0$ imply that for $\sigma \ll 1$ and $\min_{m \in 2\Z_{\le 0}} |\sigma + it + \beta - m| \ge \eps/2$, one has 
\begin{equation} \label{eq:specific-mellin-decay}
    \tilde\Phi(\sigma + it) \asymp_\eps \left(1 + \frac{|t|}{2}\right)^{\frac{\sigma + \Re\, \beta - 1}{2}} e^{-\frac{\pi}{4}|t|}
    \ll e^{-\frac{1}{2}|t|}.
\end{equation}
Since $L(s, \pi \times \tilde \pi')$ has moderate vertical growth, we may shift contours. We first shift to $\Re\, s = 1 + \eps$ to obtain
\[
\begin{aligned}
    \sum_{m=1}^\infty \Phi\left(\frac{m}{M}\right) \lambda_{\pi \times \tilde{\pi}'}(m)
    &=
    \frac{1}{2\pi i} \int_{(1+\eps)} M^s L(s, \pi \times \tilde \pi')\, \tilde\Phi(s)\, ds
    \\
    &\ll_\alpha
    M^{1+\eps} \max_{t \in \R} \frac{L(1 + \eps + it, \pi \times \tilde\pi')}{1 + |t|}.
\end{aligned}
\]
But by the convexity bound of Li \cite[Theorem 2]{li2010upper}, we have 
\begin{equation} \label{eq:bound-on-1-line}
    L(1 + \eps + it, \pi \times \tilde\pi')
    =
    L(1 + \eps, \pi[it] \times \tilde\pi')
    \ll_\eps ((1 + |t|)\, \mfC_{\pi \times \tilde\pi'})^\eps,
\end{equation}
where we recall the notation $\pi[z] = \pi \otimes |\cdot|^z$.
Thus we always have
\[
    \sum_{m=1}^\infty \Phi\left(\frac{m}{M}\right) \lambda_{\pi \times \tilde{\pi}'}(m) \ll_\eps M^{1+\eps} \mfC_{\pi \times \tilde\pi'}^{2\eps}.
\]

Now suppose that the original representations in $\mathfrak{F}_n$ are distinct, so we have $\pi[it] \neq \pi'$ for all $t \in \R$ (even after potentially twisting $\pi$ by $|\cdot|^{-iu}$ in the beginning of the proof). In this case, $\Lambda(s, \pi \times \tilde\pi')$ is entire, so $L(s, \pi \times \tilde\pi')$ is also entire; moreover, if $(ii)$ holds, then $L(s, \pi \times \tilde\pi')$ has a zero at $s = -\beta$ to cancel the corresponding pole of $\Gamma_\R(s - \mu_{\pi \times \tilde\pi',j,j'})$ inside $L_\infty(s, \pi \times \tilde\pi')$.
Then we can shift the contour in \cref{eq:mellin-rs} to $\Re\, s = -\eps$, picking up no residues in the process. Indeed, the simple pole of $\tilde \Phi(s)$ at $s = -\beta$ is either outside the contour integral (if $(i)$ holds), or cancelled by the zero of $L(s, \pi \times \tilde\pi')$ (if $(ii)$ holds). The other poles at $s + \beta \in 2\Z_{<0}$ are also outside the contour since $-\eps + \Re\, \beta > -\tfrac{1}{2} - 1 > -2$. Thus
\[
\begin{aligned}
    \sum_{m=1}^\infty \Phi\left(\frac{m}{M}\right) \lambda_{\pi \times \tilde{\pi}'}(m)
    &=
    \frac{1}{2\pi i} \int_{(-\eps)} M^s L(s, \pi \times \tilde \pi')\, \tilde\Phi(s)\, ds
    \\
    &=
    \frac{1}{2\pi i} \int_{(-\eps)} M^s L(1-s, \tilde\pi \times \pi') \frac{L(s, \pi \times \tilde\pi')}{L(1-s, \tilde\pi \times \pi')}\, \tilde\Phi(s)\, ds.
\end{aligned}
\]
Plugging in \cref{eq:bound-on-1-line}, \cref{eq:infty-quotients-rs}, and \cref{eq:specific-mellin-decay}, the triangle inequality gives
\[
    \sum_m \Phi\left(\frac{m}{M}\right) \lambda_{\pi \times \tilde{\pi}'}(m) \ll_\eps \mfC_{\pi \times \tilde\pi'}^{\frac{1}{2} + 2\eps},
\]
which completes our proof.
\end{proof}

Finally, using \cref{lem:rs-convexity} and \cref{prop:rs-nonnegative-definite}, we can prove our key estimate for exploiting the averaging over automorphic representations. This result may be of independent interest to the reader.

\begin{proposition}[Triple sums of Rankin--Selberg coefficients] \label{prop:rs-triple-sums}
Let $\mF \subset \mathfrak{F}_n$ be finite, and let $\mfC_{RS} := \max_{\pi, \pi' \in \mF} \mfC_{\pi \times \tilde \pi'}$. Let $M \gg 1$ and $(u_m)_{m \le M}$, $(\beta_\pi)_{\pi \in \mF}$ be complex sequences such that one of the following holds:
\begin{itemize}
    \item[$(i)$.] For all $\pi \in \mF$, $\beta_\pi \ll 1$, and $(u_m)$ is supported on $m \asymp M$, or
    \item[$(ii)$.] All $\pi \in \mF$ are unramified at $\infty$ and satisfy $\beta_\pi = \mu_{\pi,j}(\infty)$ for some $1 \le j \le n$.
\end{itemize}
Then for any complex weights $(w_{\pi,\pi'})_{\pi, \pi' \in \mF}$ forming a positive semi-definite matrix in $\C^{\mF \times \mF}$, one has
\begin{equation} \label{eq:rs-triple-sums-non-def}
\begin{aligned}
    \sum_{\pi, \pi' \in \mF} w_{\pi, \pi'} \sum_{m \le M} u_m\,
    \lambda_{\pi \times \tilde\pi'}(m) \left(\frac{M}{m}\right)^{\beta_\pi + \bar\beta_{\pi'}}
    &\ll
    (M\mfC_{RS})^{o(1)}\, 
    \|w\|_\infty
    |\mF|^2\, \|u\|_\infty\, M 
    \\
    &\times \left(|\mF|^{-1} + \left(1 + \frac{M}{\sqrt{\mfC_{RS}}}\right)^{-1} \right).
\end{aligned}
\end{equation}

\end{proposition}

\begin{remark}
The first line from the right-hand side of \cref{eq:rs-triple-sums-non-def} contains the `trivial' bound, which can be achieved without the averaging over $\pi, \pi' \in \mF$; the second line contains two saving factors: one from the diagonal terms $\pi = \pi'$, and one from the off-diagonal terms $\pi \neq \pi'$. If $w_{\pi,\pi'}$ are arbitrary complex weights, an application of Cauchy--Schwarz in $\pi, \pi'$ combined with the argument below produces a similar bound as in \cref{eq:rs-triple-sums-non-def}, with the diagonal saving $|\mF|^{-1}$ replaced by $|\mF|^{-1/2}$.
\end{remark}

\begin{proof}[Proof of \cref{prop:rs-triple-sums}]
Let $\mB$ denote the sum in the left-hand side of \cref{eq:rs-triple-sums-non-def}. 
By \cref{prop:rs-nonnegative-definite} and Schur's product theorem, the matrix $M \in \C^{\mF \times \mF}$ with entries
\[
    M_{\pi,\pi'} := w_{\pi,\pi'}\, \lambda_{\pi \times \tilde\pi'}(m)
\]
is also positive semi-definite. Letting $\vec{v} \in \C^{\mF}$ be the (column) vector with entries
\[
    v_{\pi} := \left(\frac{M}{m}\right)^{\bar\beta_\pi},
\]
we thus have
\[
    \sum_{\pi, \pi' \in \mF}
    w_{\pi, \pi'}\,
    \lambda_{\pi \times \tilde\pi'}(m)
    \left(\frac{M}{m}\right)^{\beta_\pi + \bar\beta_{\pi'}}
    =
    \vec{v}^{\,*} M \vec{v} \in \R_{\ge 0}.
\]
By the triangle inequality, it follows that
\[
\begin{aligned}
    |\mB| &= \left\vert \sum_{m \le M} u_m\,
    \sum_{\pi, \pi' \in \mF} w_{\pi, \pi'}\, \lambda_{\pi \times \tilde\pi'}(m) \left(\frac{M}{m}\right)^{\beta_\pi + \bar\beta_{\pi'}} \right\vert
    \\
    &\ll
    \|u\|_\infty
    \sum_{m=1}^\infty
    \Phi\left(\frac{m}{M}\right) \sum_{\pi, \pi' \in \mF} w_{\pi, \pi'}\, \lambda_{\pi \times \tilde\pi'}(m) \left(\frac{M}{m}\right)^{\beta_\pi + \bar\beta_{\pi'}},
\end{aligned}
\]
where we inserted a majorant given by the nonnegative smooth function 
\[
    \Phi(x) := x^B e^{-\sqrt{x}},
    \qquad\qquad 
    B := 
    \begin{cases} 
    1 + 2\max_{\pi \in \mF} |\Re\, \beta_\pi|, & \text{if $(i)$ holds,} \\
    0, & \text{otherwise.}
    \end{cases}
\] 
Note that having $B > 0$ is acceptable if condition $(i)$ in our assumption holds, since then $(u_m)$ is supported in $m \asymp M$. Denoting $\Phi_{\pi,\pi'}(x) := \Phi(x)\, x^{-\beta_\pi - \bar\beta_{\pi'}}$, we thus have
\[
    |\mB| \le \|u\|_\infty
    \sum_{\pi, \pi' \in \mF} w_{\pi, \pi'}
    \sum_{m=1}^\infty
    \Phi_{\pi,\pi'}\left(\frac{m}{M}\right) \lambda_{\pi \times \tilde\pi'}(m).
\]
Now the inner sum over $m$ satisfies the assumptions in \cref{lem:rs-convexity}, with $\beta = B - \beta_\pi - \bar{\beta_{\pi'}}$; indeed, if condition $(i)$ or $(ii)$ of \cref{prop:rs-triple-sums} holds, then the corresponding condition of \cref{lem:rs-convexity} holds. So applying \cref{lem:rs-convexity} yields
\[
\begin{aligned}
    |\mB| 
    &\ll 
    (M \mfC_{RS})^{o(1)} \|u\|_\infty
    \sum_{\pi, \pi' \in \mF} |w_{\pi, \pi'}|
    \left(M \one_{\pi = \pi'} + \min(M, \mfC_{RS}^{1/2})\right)
    \\
    &=
    (M \mfC_{RS})^{o(1)} \|u\|_\infty\, M
    \left(
    \sum_{\pi \in \mF} |w_{\pi, \pi}|
    +
    \min\left(1, \frac{\sqrt{\mfC_{RS}}}{M}\right) \sum_{\pi, \pi' \in \mF} |w_{\pi,\pi'}|\right).
\end{aligned}
\]
Trivially bounding $\frac{1}{|\mF|} \sum_{\pi \in \mF} |w_{\pi,\pi}|$ and $\frac{1}{|\mF|^2} \sum_{\pi, \pi' \in \mF} |w_{\pi,\pi'}|$ by the $\ell^\infty$ norm completes our proof.
\end{proof}

\section{The density theorems} \label{sec:density-theorems}

\subsection{The non-archimedean case of \cref{thm:density}} \label{subsec:non-archimedean}

Fix $n \ge 2$ and a place $v = p < \infty$. We aim for a density theorem for the local parameters of cuspidal automorphic representations at $p$, which we access through the Dirichlet coefficients at powers of $p$. The following power sum bound due to Tur\'an will be helpful in this process.

\begin{lemma}[Tur\'an's second theorem for power sums \cite{montgomery1994ten}] \label{lem:turan}
For any positive integers $M, N$ and any complex numbers $z_1, \ldots, z_N$ with $\max_j |z_j| \ge 1$, one has
\[
    \max_{M+1 \le k \le M+N}
    \left\vert \sum_{j=1}^N z_j^k \right\vert 
    \gg_N 
    \frac{1}{M^N} \max_j |z_j|^M.
\]
\end{lemma}

\begin{proof}
By symmetry, we may assume without loss of generality that $|z_1| \ge |z_2| \ge \cdots \ge |z_N|$, so in particular $|z_1| \ge 1$. By writing
\[
\begin{aligned}
    \max_{M+1 \le k \le M+N}
    \left\vert \sum_{j=1}^N z_j^k \right\vert 
    &=
    \max_{M+1 \le k \le M+N} |z_1|^k
    \left\vert \sum_{j=1}^N (z_1^{-1} z_j)^k \right\vert 
    \\
    &\ge 
    |z_1|^M \max_{M+1 \le k \le M+N}
    \left\vert \sum_{j=1}^N (z_1^{-1} z_j)^k \right\vert,
\end{aligned}
\]
we can reduce to the case $|z_1| = 1$. Now Theorem 4 from \cite[Ch.\,5, p.\ 94]{montgomery1994ten} states the sharper and more general lower bound
\[
    \max_{M+1 \le k \le M+N}
    \left\vert \sum_{j=1}^N b_j z_j^k \right\vert 
    \ge 
    2\left(\frac{N}{8e(M+N)}\right)^N \min_{j=1}^N \sum_{i=1}^j b_i
    \qquad\quad 
    \text{when } 1 = |z_1| \ge |z_2| \ge \cdots \ge |z_N|,
\]
for any complex numbers $b_1, \ldots b_N$. We take $b_1 = \cdots = b_N = 1$ and apply a crude lower bound on the right-hand side.
\end{proof}

\begin{theorem}[Density at finite places] \label{thm:density-finite}
Let $\mF \subset \mathfrak{F}_n$ be finite such that
\[
    \forall \pi \in \mF,\qquad\quad \pi \text{ is unramified at $v$ and } \max_j |\Re\, \mu_{\pi, j}(p)| \ge \theta,
\]
for some $\theta \in (0, \tfrac{1}{2})$. Then for $\mfC_{RS} := \max_{\pi, \pi' \in \mF} \mfC_{\pi \times \tilde\pi'}$, one has
\[
    |\mF| \ll_{n,p,\theta} \mfC_{RS}^{\frac{1-2\theta}{4\theta} + o(1)}.
\]
\end{theorem}

\begin{proof}
Let $\ell$ be a positive integer and $(w_\pi)_{\pi \in \mF}$ be complex numbers of absolute value $1$ to be chosen shortly. Consider the sum
\begin{equation} \label{eq:sum-finite-places}
    \mS := \frac{1}{|\mF|^2} \sum_{\pi, \pi' \in \mF} w_\pi \bar w_{\pi'} 
    \lambda_{\pi \times \tilde\pi'}(\ell).
\end{equation}
By \cref{prop:rs-triple-sums} for $M = \ell$ and the sequence 
\[
    u_m := \one_{m = \ell},
\] 
we obtain the upper bound
\[
    \mS \ll (\ell \mfC_{RS})^{o(1)} \left(\frac{\ell}{|\mF|} + \sqrt{\mfC_{RS}} \right).
\]
Now take $\ell = p^k$, for some $k \ge 1$. By \cref{prop:rs-nonnegative-definite}, we can lower bound
\[
    \mS
    \ge 
    \frac{1}{k |\mF|^2} \left\vert \sum_{\pi \in \mF} w_\pi \sum_{j=1}^n \alpha_{\pi,j}(p)^k \right\vert^2.
\]
Picking $w_\pi$ to achieve absolute values around the inner sum above, it follows that
\[
    \frac{1}{\sqrt{k} |\mF|} \sum_{\pi \in \mF} \left\vert\sum_{j=1}^n \alpha_{\pi,j}(p)^k \right\vert
    \ll 
    (p^k \mfC_{RS})^{o(1)} \sqrt{\frac{p^k}{|\mF|} + \sqrt{\mfC_{RS}}}.
\]
This holds for any $k \ge 1$. Summing over $k$ from $k_0 + 1$ to $k_0 + n$, for some $k_0 \ge 1$ to be chosen shortly, we obtain
\[
    \frac{1}{\sqrt{k_0} |\mF|}
    \sum_{\pi \in \mF} \sum_{k=k_0+1}^{k_0+n} \left\vert\sum_{j=1}^n \alpha_{\pi,j}(p)^k \right\vert
    \ll_p 
    (p^{k_0} \mfC_{RS})^{o(1)} \sqrt{\frac{p^{k_0}}{|\mF|} + \sqrt{\mfC_{RS}}},
\]
where we implicitly used that $n$ is fixed. Applying \cref{lem:turan}, we reach
\begin{equation} \label{eq:after-turan}
    \frac{1}{k_0^{n+1/2} |\mF|} \sum_{\pi \in \mF} \max_j |\alpha_{\pi,j}(p)|^{k_0}
    \ll_p 
    (p^{k_0} \mfC_{RS})^{o(1)} \sqrt{\frac{p^{k_0}}{|\mF|} + \sqrt{\mfC_{RS}}}.
\end{equation}
But by our assumption on $\mF$ and unitarity ($\{\mu_{\pi,j}\} = \{-\bar \mu_{\pi,j}\}$), we have 
\[
    \max_j |\alpha_{\pi,j}(p)|^{k_0}
    =
    \max_j p^{k_0 \Re\, \mu_{\pi,j}(p)}
    \ge 
    p^{k_0 \theta},
    \qquad\quad 
    \forall \pi \in \mF.
\]
Plugging this into \cref{eq:after-turan} and squaring, we reach
\begin{equation} \label{eq:after-turan-2}
    \frac{1}{k_0^{2n+1}} p^{2k_0\theta}
    \ll_p 
    (p^{k_0} \mfC_{RS})^{o(1)} \left(\frac{p^{k_0}}{|\mF|} + \sqrt{\mfC_{RS}} \right)
\end{equation}
To optimize, we pick $k_0$ such that $p^{k_0} \asymp_p |\mF| \sqrt{\mfC_{RS}}$; in particular, the factor of $k_0^{2n+1}$ grows like $(|\mF| \mfC_{RS})^{o(1)}$. We conclude that
\[
    (|\mF|\sqrt{\mfC_{RS}})^{2\theta} \ll_p (|\mF|\mfC_{RS})^{o(1)} \sqrt{\mfC_{RS}},
\]
which rearranges to the desired bound.
\end{proof}

In particular, \cref{thm:density-finite} establishes \cref{thm:density} at the finite places (using \cref{eq:total-conductor-bound} to bound $\mfC_{RS}$).


\subsection{The archimedean case of \cref{thm:density}} \label{subsec:archimedean}
Fix $n \ge 2$ and $v = \infty$. Here we modify the argument in \cref{subsec:non-archimedean} to prove density theorems for the local parameters of cuspidal automorphic representations at $\infty$. Following Luo--Rudnick--Sarnak \cite{luo1995selberg}, we access these parameters through the vanishing of $L(s, \pi \times \tilde\pi')$ at $s = \mu_{\pi,j}(\infty) + \bar \mu_{\pi,j'}(\infty)$; this was used in \cref{lem:rs-convexity,prop:rs-triple-sums}, conditions $(ii)$.

\begin{theorem}[Density at the infinite place] \label{thm:density-infinity}
Let $\mF \subset \mathfrak{F}_n$ be finite such that
\[
    \forall \pi \in \mF,\qquad\quad \pi \text{ is unramified at $v$ and } \max_j |\Re\, \mu_{\pi, j}(\infty)| \ge \theta,
\]
for some $\theta \in (0, \tfrac{1}{2})$. Then with $\mfC_{RS} := \max_{\pi, \pi' \in \mF} \mfC_{\pi \times \tilde\pi'}$, one has
\[
    |\mF| \ll_{n,\theta} \mfC_{RS}^{\frac{1-2\theta}{4\theta} + o(1)}.
\]
\end{theorem}

\begin{proof}
For $\pi \in \mF$, let
\[
    \beta_\pi := \max_j \Re\, \mu_{\pi,j}(\infty).
\]
In particular, by unitarity we have $\beta_\pi = \max_j |\Re\, \mu_{\pi,j}(\infty)| \ge \theta$. Let $\ell$ be a positive integer and $(w_\pi)_{\pi \in \mF}$ be $1$-bounded complex weights to be chosen shortly.
In analogy with \cref{eq:sum-finite-places}, we consider the sum
\[
\begin{aligned}
    \mS &:=
    \frac{1}{|\mF|^2} \sum_{\pi, \pi' \in \mF} w_\pi \bar w_{\pi'} \ell^{\beta_\pi + \bar \beta_{\pi'}}
    \\
    &= \frac{1}{|\mF|^2}\, \sum_{\pi, \pi' \in \mF}
    w_\pi \bar w_{\pi'}
    \lambda_{\pi \times \tilde\pi'}(1) \left(\frac{\ell}{1}\right)^{\beta_\pi + \bar\beta_{\pi'}}.
\end{aligned}
\]
Using \cref{prop:rs-triple-sums} (specifically, \cref{eq:rs-triple-sums-non-def}) for $M = \ell$ and the sequence
\[
    u_m := \one_{m = 1},
\]
we can again bound
\[
    \mS \ll (\ell \mfC_{RS})^{o(1)} \left(\frac{\ell}{|\mF|} + \sqrt{\mfC_{RS}} \right).
\]
On the other hand, picking $w_\pi := |\ell^{\beta_\pi}| / \ell^{\beta_\pi}$, we have the lower bound
\[
    \ell^{2\theta} \le 
    \frac{1}{|\mF|^2} \sum_{\pi, \pi' \in \mF} \left\vert\ell^{\beta_\pi} \ell^{\bar\beta_{\pi'}}\right\vert
    =
    \mS,
\]
Putting these two bounds together gives
\begin{equation} \label{eq:optimize-archimedean}
    \ell^{2\theta}
    \ll 
    (\ell \mfC_{RS})^{o(1)} \left(\frac{\ell}{|\mF|} + \sqrt{\mfC_{RS}} \right).
\end{equation}
We pick $\ell \asymp |\mF| \sqrt{\mfC_{RS}}$ to optimize, and conclude that
\[
    \left(|\mF| \sqrt{\mfC_{RS}}\right)^{2\theta - o(1)}
    \ll 
    \mfC_{RS}^{\frac{1}{2}+o(1)},
\]
which rearranges to the desired bound, as before.
\end{proof}

This also completes the proof of \cref{thm:density}, in light of \cref{thm:density-finite} and \cref{eq:total-conductor-bound}.

\subsection{Amplification and the pointwise Luo--Rudnick--Sarnak bounds} \label{subsec:amplif-unram}

The proofs in \cref{subsec:non-archimedean,subsec:archimedean} might seem a bit wasteful, since they rely on upper-bounding Rankin--Selberg coefficients at a single positive integer $m$ (either $m = \ell$ or $m = 1$) by a sum of such coefficients over a long interval of size $M$. When $|\mF| = 1$, this reduces to the global arguments outlined in \cref{eq:simple-argument-finite,eq:simple-argument-infinite}, which recover the Jacquet--Shalika pointwise  bounds $|\Re\, \mu_{\pi,j}| \le \frac{1}{2}$.

The Landau--Serre method \cite{serreletter,landau1915anzahl} improved the Jacquet--Shalika pointwise bounds at finite places, using archimedean twists by $m^{it}$ with $|t| \le T$ (for some parameter $T$ to be optimized). Serre's original argument \cite{serreletter} relied on a theorem of Landau \cite{landau1915anzahl} for sums of Dirichlet coefficients with sharp cutoffs; a variant of this argument \cite[\S 7.1]{blomer2013role} uses a smooth sum over a short interval of length $M/T$, effectively amplifying the contribution of $m = \ell$ by a factor of $T$ compared to \cref{eq:simple-argument-finite}. One can then apply the (approximate) functional equation for each $L(s, \pi[it] \times \tilde\pi)$ rather than $L(s, \pi \times \tilde \pi)$. This decreases the main-term residue contribution by a factor of $T$, but increases the analytic conductor and the length of the dual sum by roughly $T^{n^2}$; notably, one 
can win an additional factor of $T^{1/2}$ from the $t$-oscillation of gamma factors in this process. Taking $\ell = p^k$ to be a large prime power, this leads to the bound $|\Re\, \mu_{\pi,j}(p)| \le \tfrac{1}{2} - \tfrac{1}{n^2+1}$ at unramified finite places.

There is also a non-archimedean version of this argument \cite[\S 7.1]{blomer2011ramanujan} based on an idea of Iwaniec \cite{iwaniec1996lowest}, which replaces a short interval of length $M/T$ by an arithmetic progression of modulus $q$, and $m^{it}$ (for $|t| \le T$) by $\chi(m)$ (for Dirichlet characters $\chi \mod q$). The saving of $T^{1/2}$ from the vertical oscillation of the gamma factors is then replaced by a gain of $q^{1/2}$ from Deligne's bound for hyper-Kloosterman sums modulo $q$. This produces bounds of the same strength as the Landau--Serre method at the finite places, but has the advantage of working equally well at the infinite place \cite{luo1995selberg}. 

Now, one can apply the same (archimedean or non-archimedean) amplification trick in the setting of a finite family $\mF \subset \mathfrak{F}_n$. In particular, rather than using a full interval of integers $m \sim M \approx \ell$ in \cref{subsec:non-archimedean}, one can bound
\[
    \sum_{\pi, \pi' \in \mF} w_\pi \bar w_{\pi'} \lambda_{\pi \times \tilde\pi'}(\ell)
    \le 
    \sum_{\substack{m \sim M \\ |m - \ell| < \frac{M}{T}}}
    \sum_{\pi, \pi' \in \mF} w_\pi \bar w_{\pi'} \lambda_{\pi \times \tilde\pi'}(m)
\]
or alternatively
\[
    \sum_{\pi, \pi' \in \mF} w_\pi \bar w_{\pi'} \lambda_{\pi \times \tilde\pi'}(\ell)
    \le
    \sum_{\substack{m \sim M \\ m \equiv \pm \ell \pmod{q}}}
    \sum_{\pi, \pi' \in \mF} w_\pi \bar w_{\pi'} \lambda_{\pi \times \tilde\pi'}(m).
\]
It turns out that this strategy does not improve our main result from \cref{thm:density} for general families $\mF$, but it can at least reconcile it with the pointwise bounds of Luo--Rudnick--Sarnak \cite{luo1995selberg}. The result of (the non-archimedean version of) this argument is given below.

\begin{proposition} \label{prop:density-amplif} Let $v$ be a place of $\Q$, $\mfC \ge 1$, and $\mF \subset \mathfrak{F}_n$ such that all $\pi \in \mF$ are unramified at $v$ and have total conductors $\mfC_\pi \le \mfC$ (as defined in \cref{eq:total-conductor}). Let $q \neq v$ be either $1$ or a prime with $q \nmid \prod_{\pi \in \mF} \mf_\pi$. Then for any $\theta \in (0, \tfrac{1}{2})$, one has 
\[
    \#\left\{ \pi \in \mF : \max_j |\Re\, \mu_{\pi,j}(v)| \ge \theta\right\} \ll_{v,\theta} \frac{1}{q} \left(\mfC^{2n}\, q^{n^2-1}\right)^{\frac{1-2\theta}{4\theta} + o(1)}.
\]
In fact, one can replace the factor of $\mfC^{2n}$ in the right-hand side with $\mfC_{RS} := \max_{\pi, \pi' \in \mF} \mfC_{\pi \times \tilde\pi'}$.
\end{proposition}

Before proving \cref{prop:density-amplif}, we remark that it is precisely equivalent to the combination of our density result from \cref{thm:density} and the best known pointwise bounds. Indeed, the optimal choice of the parameter $q$ in \cref{prop:density-amplif} will be either $q = 1$ (if the exponent of $q$ in the right-hand side is nonnegative) or $q \to \infty$ (if the exponent of $q$ is negative). The first case recovers our \cref{thm:density}, while the second case recovers the Luo--Rudnick--Sarnak bound \cite{luo1995selberg,luo1999generalized}.

\begin{corollary} \label{cor:LRS}
For any $\pi \in \mathfrak{F}_n$ which is unramified at a place $v$, $\max_j |\Re\, \mu_{\pi,j}(v)| \le \tfrac{1}{2} - \tfrac{1}{n^2+1}$.
\end{corollary}

\begin{proof}[Proof of \cref{cor:LRS} assuming \cref{prop:density-amplif}]
Let $\eps > 0$ and apply \cref{prop:density-amplif} for $\mF = \{\pi\}$ and $\theta := \tfrac{1}{2} - \tfrac{1}{n^2+1} + \eps$. Then the quantity
\begin{equation} \label{eq:rewrite-q-expo}
    (n^2-1)\frac{1-2\theta}{4\theta} - 1
    =
    \frac{n^2+1}{2\theta}\left(\frac{1}{2} - \frac{1}{n^2+1} - \theta \right)
\end{equation}
is strictly negative,
and one can take $q \to \infty$ to conclude that $\#\left\{ \pi \in \mF : \max_j |\Re\, \mu_{\pi,j}(v)| \ge \theta\right\} = 0$. It follows that $\max_j |\Re\, \mu_{\pi,j}(v)| < \tfrac{1}{2} - \tfrac{1}{n^2+1} + \eps$ for any $\eps > 0$.
\end{proof}

\begin{remark}
The original argument of Luo--Rudnick--Sarnak \cite{luo1995selberg} establishes \cref{cor:LRS} using an average over moduli $q \sim Q$, following a trick of Duke--Iwaniec \cite{duke1989estimates}. This removes the need for having positive Dirichlet coefficients (and works equally well for symmetric square $L$-functions \cite[Appendix 2]{kim2003functoriality}), but it is also less compatible with our approach (which fundamentally relies on a positivity property for families of Rankin--Selberg $L$-functions). The fact that a single modulus $q$ can be used to recover \cref{cor:LRS} was also noted by Blomer--Brumley \cite{blomer2013nonvanishing}.
\end{remark}

The proof of \cref{prop:density-amplif} requires a twisted analogue of \cref{lem:rs-convexity}. To state this, we construct a majorant for the indicator function of a residue class $r \in (\Z/q\Z)^\times$, where $q$ is prime, by
\begin{equation} \label{eq:cq-majorant}
\begin{aligned}
    c_q(m, r) := \frac{2}{q-1}\Bigg(1 + \sum_{\substack{\chi \pmod{q} \\ \text{even primitive}}} \bar\chi(r) \chi(m) \Bigg)&
    \\
    \ge 
    \frac{2}{q-1}
    \sum_{\substack{\chi \pmod{q} \\ \text{even}}} \bar\chi(r) \chi(m)& 
    =
    \one_{m \equiv \pm r \pmod{q}}.
\end{aligned}
\end{equation}

\begin{lemma}[Smooth sums of Rankin--Selberg coefficients with unramified twists] \label{lem:rs-convexity-modified}
Assume the setup of \cref{lem:rs-convexity}. Let $q \nmid \mf_\pi \mf_{\pi'}$ be a prime and $r \in (\Z/q\Z)^\times$. Then one has
\[
    \sum_{m=1}^\infty c_q(m,r)\, \Phi\left(\frac{m}{M}\right) \lambda_{\pi \times \tilde{\pi}'}(m)
    \ll
    (Mq \mfC_{\pi \times \tilde\pi'})^{o(1)}
    \left(
    \frac{M}{q} \one_{\pi = \pi'}
    +
    \sqrt{\mfC_{\pi \times \tilde\pi'} q^{n^2-1}}\right).
\]
\end{lemma}

\begin{proof}
After expanding $c_q(m,r)$, the left-hand side becomes
\[
    \frac{2}{q-1} \left(\sum_{m=1}^\infty \Phi\left(\frac{m}{M}\right) \lambda_{\pi \times \tilde{\pi}'}(m)
    +
    \sum_{\substack{\chi \pmod{q} \\ \text{even primitive}}} \bar\chi(r)
    \sum_{m=1}^\infty \Phi\left(\frac{m}{M}\right) \lambda_{(\pi \otimes \chi) \times \tilde{\pi}'}(m)
    \right).
\]
The first sum is bounded appropriately by \cref{lem:rs-convexity}. For the second sum, we follow the proof of \cref{lem:rs-convexity} with $L(s, \pi \times \tilde \pi')$ replaced by $L(s, (\pi \otimes \chi) \times \tilde\pi')$. Note that there are no poles involved, since $q \mid \mf_{\pi \otimes \chi}$ (by \cref{eq:twisted-epsilon}), but $q \nmid \mf_{\pi'}$. After shifting the contour of integration, rather than using \cref{eq:infty-quotients-rs} directly, we also use \cref{lem:twisted-epsilon-rs} to write
\[
\begin{aligned}
    \sum_{\substack{\chi \pmod{q} \\ \text{even prim.}}} \bar\chi(r)
    \frac{L(s,(\pi \otimes \chi) \times \tilde\pi')}{L(1-s,(\tilde\pi \otimes \bar\chi) \times \pi'))}
    =
    \sum_{\substack{\chi \pmod{q} \\ \text{even prim.}}} \bar\chi(r)
    \epsilon(s, (\pi \otimes \chi) \times \tilde\pi')
    \frac{L_\infty(1-s, (\tilde\pi \otimes \bar\chi) \times \pi')}{L_\infty(s,(\pi \otimes \chi) \times \tilde\pi')}
    \\
    = 
    q^{-n^2s} \eta_\pi(q) \bar\eta_{\pi'}(q)
    \epsilon(s, \pi \times \tilde\pi')
    \frac{L_\infty(1-s, \tilde\pi \times \pi')}{L_\infty(s, \pi \times \tilde\pi')}
    \sum_{\substack{\chi \pmod{q} \\ \text{even primitive}}}
    \bar\chi(r) \tau_\chi^{n^2} \chi(\mf_{\pi \times \tilde\pi'})
    \\
    \ll 
    (1 + |s|)^{O(1)} \left(\mfC_{\pi \times \tilde\pi'} q^{n^2}\right)^{\frac{1}{2}-\Re\, s}
    \left\vert 
    \sum_{\substack{\chi \pmod{q} \\ \text{even prim.}}}
    \bar\chi(r) \left(\frac{\tau_\chi}{\sqrt{q}}\right)^{n^2} \chi(\mf_{\pi \times \tilde\pi'})
    \right\vert.
\end{aligned}
\]
One then achieves square-root cancellation ($\ll \sqrt{q}$) in the sum over Dirichlet characters by plugging in \cite[(3.18)]{luo1995selberg}, which stems from Deligne's bound for hyper-Kloosterman sums. When $\Re\, s = -\eps$, we obtain
\[
    \frac{2}{q-1} \sum_{\substack{\chi \pmod{q} \\ \text{even prim.}}} \bar\chi(r)
    \frac{L(s,(\pi \otimes \chi) \times \tilde\pi')}{L(1-s,(\tilde\pi \otimes \bar\chi) \times \pi'))}
    \ll_\eps (1 + |s|)^{O(1)} \left(\mfC_{\pi \times \tilde\pi'} q^{n^2-1} \right)^{\tfrac{1}{2}+\eps},
\]
which leads to the desired bound.
\end{proof}

\begin{proof}[Proof of \cref{prop:density-amplif}]
We focus on the case $q > 1$, since the case $q = 1$ is simpler (and was already covered by \cref{thm:density}). Let $r \in (\Z/q\Z)^\times$.

The only significant change from the proof of \cref{thm:density} is that we use \cref{lem:rs-convexity-modified} instead of \cref{lem:rs-convexity}. This directly leads to a modified version of \cref{prop:rs-triple-sums} for families of representations which are unramified at $q$, where the conclusion \cref{eq:rs-triple-sums-non-def} is replaced by
\begin{equation} \label{eq:rs-triple-sums-non-def-modified}
\begin{aligned}
    \sum_{\pi, \pi' \in \mF} w_{\pi, \pi'} \sum_{m \le M} u_m\, c_q(m, r)
    \lambda_{\pi \times \tilde\pi'}(m) \left(\frac{M}{m}\right)^{\beta_\pi + \bar\beta_{\pi'}}
    &\ll
    (Mq\mfC_{RS})^{o(1)}\, 
    \|w\|_\infty
    |\mF|^2\, \|u\|_\infty\, M 
    \\
    &\times \left(q^{-1}|\mF|^{-1} + M^{-1}\sqrt{\mfC_{RS} q^{n^2-1}} \right).
\end{aligned}
\end{equation}
We then use the bound \cref{eq:rs-triple-sums-non-def-modified} in place of \cref{eq:rs-triple-sums-non-def} in the proofs from \cref{subsec:non-archimedean} (resp., \cref{subsec:archimedean}), for the same sequences $(u_m)$ supported on $m = \ell$ (resp., $m = 1$), and the choice $r = \ell$ (resp., $r = 1$). The bounds \cref{eq:after-turan-2,eq:optimize-archimedean} now become
\begin{equation} \label{eq:optimize-amplif}
\begin{aligned}
    \frac{1}{k_0^{2n+1}} p^{2k_0\theta}
    &\ll_p 
    (p^{k_0} q\mfC_{RS})^{o(1)} \left(\frac{p^{k_0}}{q|\mF|} + \sqrt{\mfC_{RS} q^{n^2-1}} \right),
    \\
    \ell^{2\theta}
    &\ll 
    (\ell q\mfC_{RS})^{o(1)} \left(\frac{\ell}{q|\mF|} +  \sqrt{\mfC_{RS} q^{n^2-1}} \right),
\end{aligned}
\end{equation}
respectively. Finally, choosing $p^{k_0} \asymp_p q|\mF| \sqrt{\mfC_{RS} q^{n^2-1}} \asymp \ell$ completes the proof of \cref{prop:density-amplif}.
\end{proof}

\begin{remark}
When $v = p < \infty$, the dependency of the implied constant in \cref{thm:density} on $p$ may be important for some applications. This dependency arises partly through the optimization in $\ell = p^k$ (which only allows for multiplicative jumps of size $p$), and could be improved when the parameter $q$ is also available for optimization as in \cref{eq:optimize-amplif}.
\end{remark}

\subsection{Amplification at ramified primes and the proof of \cref{thm:density-blomer-family}} \label{subsec:amplif-ram}
The proof of
\cref{prop:density-amplif} used twists by characters mod $q$, where $q$ is an unramified prime for all $\pi \in \mF$. For special families, it can be fruitful to use twists at the ramified primes instead, due to a conductor stability phenomenon. This leads to the following bound, which effectively removes the factor $\mfC^{2n}$ from the upper bound in \cref{prop:density-amplif}.

\begin{proposition} \label{prop:density-amplif-ramified} Let $v$ be a place of $\Q$, $q \neq v$ be a prime, $I \subset [0, \infty)$ be a fixed compact interval, and $\mF_I(q)$ be the family from \cref{thm:density}. Then for any $\theta \in (0, \tfrac{1}{2})$, one has 
\[
    \#\left\{ \pi \in \mF_I(q) : \max_j |\Re\, \mu_{\pi,j}(v)| \ge \theta\right\} \ll_{v,I,\theta} \frac{1}{q} \left(q^{n^2-1}\right)^{\frac{1-2\theta}{4\theta} + o(1)}.
\]
\end{proposition}

Like \cref{prop:density-amplif}, \cref{prop:density-amplif-ramified} requires a modified version of \cref{lem:rs-convexity}.

\begin{lemma}[Smooth sums of Rankin--Selberg coefficients with ramified twists] \label{lem:rs-convexity-modified-ramified}
Let $q$ be a prime, and assume the setup of \cref{lem:rs-convexity} for some $\pi, \pi' \in \mF_I(q)$. Let $r \in (\Z/q\Z)^\times$. Then with the notation from \cref{eq:cq-majorant}, one has
\[
    \sum_{m=1}^\infty c_q(m,r)\, \Phi\left(\frac{m}{M}\right) \lambda_{\pi \times \tilde{\pi}'}(m)
    \ll
    (Mq)^{o(1)}
    \left(
    \frac{M}{q}\one_{\pi = \pi'}
    +
    \sqrt{q^{n^2-1}}\right).
\]
\end{lemma}

\begin{proof}
As in the proof of \cref{lem:rs-convexity-modified}, we expand the left-hand side and reduce to establishing the two bounds
\begin{equation} \label{eq:required-bound-ram-twist-1}
    \frac{2}{q-1} \sum_{m=1}^\infty \Phi\left(\frac{m}{M}\right) \lambda_{\pi \times \tilde{\pi}'}(m)
    \ll
    (Mq)^{o(1)}
    \left(
    \frac{M}{q}\one_{\pi = \pi'}
    +
    \sqrt{q^{n^2-1}}\right),
\end{equation}
\begin{equation} \label{eq:required-bound-ram-twist-2}
    \frac{2}{q-1} \sum_{\substack{\chi \pmod{q} \\ \text{even primitive}}} \bar\chi(r)
    \sum_{m=1}^\infty \Phi\left(\frac{m}{M}\right) \lambda_{(\pi \otimes \chi) \times \tilde{\pi}'}(m)
    \ll (Mq)^{o(1)} \sqrt{q^{n^2-1}}.
\end{equation}
The claim in \cref{eq:required-bound-ram-twist-1} follows from \cref{lem:rs-convexity} and the bounds
\begin{equation} \label{eq:RS-conductor-blomer}
    \mfC_{\pi \times \tilde\pi'} \ll \mf_{\pi \times \tilde\pi'} \le q^{2n-2},
\end{equation}
which use the fact that all local parameters of $\pi, \pi'$ at $\infty$ are $O(1)$, as well as the improved Bushnell--Henniart bounds for Rankin--Selberg conductors when $\omega_\pi \bar\omega_{\pi'}$ is unramified \cite[Appendix B]{brumley2022zeros}.

To establish \cref{eq:required-bound-ram-twist-2}, we follow the proof of \cref{lem:rs-convexity} with $L(s, (\pi \otimes \chi) \times \tilde\pi')$ in place of $L(s, \pi \times \tilde \pi')$. Once again, there are no poles involved, since $(\pi \otimes \chi)[it] \neq \pi'$ for all $t \in \R$ (this can be seen, e.g., by comparing arithmetic conductors of the Rankin--Selberg convolution with $\tilde\pi'$: one has $\mf_{(\pi \otimes \chi) \times \tilde\pi'} = q^{n^2}$ by \cref{lem:twisted-epsilon-rs-ram}, but $\mf_{\pi' \times \tilde\pi'} \le q^{2n-2}$ by \cite[Appendix B]{brumley2022zeros}). After shifting contours, rather than using \cref{eq:infty-quotients-rs}, we use \cref{lem:twisted-epsilon-rs-ram} and the fact that all local parameters of $\pi, \pi'$ at $\infty$ are $O(1)$ to write
\[
\begin{aligned}
    \sum_{\substack{\chi \pmod{q} \\ \text{even prim.}}} \bar\chi(r)
    \frac{L(s,(\pi \otimes \chi) \times \tilde\pi')}{L(1-s,(\tilde\pi \otimes \bar\chi) \times \pi'))}
    =
    \sum_{\substack{\chi \pmod{q} \\ \text{even prim.}}} \bar\chi(r)
    \epsilon(s, (\pi \otimes \chi) \times \tilde\pi')
    \frac{L_\infty(1-s, (\tilde\pi \otimes \bar\chi) \times \pi')}{L_\infty(s,(\pi \otimes \chi) \times \tilde\pi')}
    \\
    = 
    q^{-n^2s} \tau_{\pi \times \tilde\pi'}(\infty)
    \frac{L_\infty(1-s, \tilde\pi \times \pi')}{L_\infty(s, \pi \times \tilde\pi')}
    \sum_{\substack{\chi \pmod{q} \\ \text{even primitive}}}
    \bar\chi(r) \tau_\chi^{n^2}
    \\
    \ll 
    (1 + |s|)^{O(1)} \left(q^{n^2}\right)^{\frac{1}{2}-\Re\, s}
    \left\vert 
    \sum_{\substack{\chi \pmod{q} \\ \text{even prim.}}}
    \bar\chi(r) \left(\frac{\tau_\chi}{\sqrt{q}}\right)^{n^2}
    \right\vert.
\end{aligned}
\]
By \cite[(3.18)]{luo1995selberg} (which stems from Deligne's bound for hyper-Kloosterman sums), the inner sum is $O(\sqrt{q})$. When $\Re\, s = -\eps$, we conclude that
\[
    \frac{2}{q-1} \sum_{\substack{\chi \pmod{q} \\ \text{even prim.}}} \bar\chi(r)
    \frac{L(s,(\pi \otimes \chi) \times \tilde\pi')}{L(1-s,(\tilde\pi \otimes \bar\chi) \times \pi'))}
    \ll_\eps (1 + |s|)^{O(1)} \left(q^{n^2-1} \right)^{\tfrac{1}{2}+\eps},
\]
which leads to \cref{eq:required-bound-ram-twist-2} as in the proof of \cref{lem:rs-convexity}.
\end{proof}

\begin{proof}[Proof of \cref{prop:density-amplif-ramified}]
Let $r \in (\Z/q\Z)^\times$, and recall from \cref{eq:cq-majorant} that $c_q(m, r)$ is a majorant for the indicator function $\one_{m \equiv \pm r \pmod{q}}$.

We follow the proof of \cref{thm:density} using \cref{lem:rs-convexity-modified-ramified} instead of \cref{lem:rs-convexity}. This leads to a modified version of \cref{prop:rs-triple-sums} for families $\mF \subset \mF_I(q)$, where \cref{eq:rs-triple-sums-non-def} is replaced by
\begin{equation} \label{eq:rs-triple-sums-non-def-modified-ramified}
\begin{aligned}
    \sum_{\pi, \pi' \in \mF} w_{\pi, \pi'} \sum_{m \le M} u_m\, c_q(m, r)
    \lambda_{\pi \times \tilde\pi'}(m) \left(\frac{M}{m}\right)^{\beta_\pi + \bar\beta_{\pi'}}
    &\ll
    (Mq)^{o(1)}\, 
    \|w\|_\infty
    |\mF|^2\, \|u\|_\infty\, M 
    \\
    &\times \left(q^{-1}|\mF|^{-1} + M^{-1}\sqrt{ q^{n^2-1}} \right).
\end{aligned}
\end{equation}

We then use  \cref{eq:rs-triple-sums-non-def-modified-ramified} in place of \cref{eq:rs-triple-sums-non-def} in the proofs from \cref{subsec:non-archimedean} (resp., \cref{subsec:archimedean}), for the same sequences $(u_m)$ supported on $m = \ell$ (resp., $m = 1$), and $r = \ell$ (resp., $r = 1$). The bounds in \cref{eq:after-turan-2,eq:optimize-archimedean} become
\[
\begin{aligned}
    \frac{1}{k_0^{2n+1}} p^{2k_0\theta}
    &\ll_p 
    (p^{k_0}q)^{o(1)} \left(\frac{p^{k_0}}{q|\mF|} + \sqrt{q^{n^2-1}} \right),
    \\
    \ell^{2\theta}
    &\ll 
    (\ell q)^{o(1)} \left(\frac{\ell}{q|\mF|} +  \sqrt{q^{n^2-1}} \right),
\end{aligned}
\]
respectively. Choosing $p^{k_0} \asymp_p q|\mF| \sqrt{q^{n^2-1}} \asymp \ell$ completes the proof of \cref{prop:density-amplif-ramified}.
\end{proof}

We can finally establish \cref{thm:density-blomer-family}.

\begin{proof}[Proof of \cref{thm:density-blomer-family}]
The first bound in \cref{eq:density-blomer-family-1} (using the first term from the minimum) is simply \cref{thm:density} combined with \cref{eq:RS-conductor-blomer}. The second bound in \cref{eq:density-blomer-family-1} (using the second term from the minimum) is just \cref{prop:density-amplif-ramified}, after rewriting the exponent of $q$ as in \cref{eq:rewrite-q-expo}.
\end{proof}

\textbf{Acknowledgements.}
The authors would like to thank James Maynard for many fruitful conversations, as well as Valentin Blomer, Farrell Brumley, Daniel Bump, Peter Humphries, Peter Sarnak, Kannan Soundararajan, Jesse Thorner, and the referees, for helpful comments and suggestions. We are also very grateful to Edgar Assing for detailed discussions on the local computation from \cref{subsec:twist-ram-chars}. The first author was supported by an NSF Postdoctoral Fellowship. The second author was supported by an EPSRC Scholarship, as well as through EXC-2047/1-390685813 and by ERC Advanced Grant 101054336.

\appendix

\section{Some local computations} \label{sec:ramified-primes}

The computations in this appendix are based on the Bernstein--Zelevinsky classification of generic, unitary, irreducible, smooth representations of $\GL_n(\Q_p)$ where $p$ is a prime \cite[\S 14.6]{goldfeld2011automorphicII}, and on the corresponding expressions for local factors of Rankin--Selberg $L$-functions due to Jacquet, Piatetskii-Shapiro and Shalika \cite{jacquet1983rankin}.

\subsection{Positive semi-definiteness} \label{subsec:pos-def-ram-primes}
Here we complete the proof of \cref{prop:rs-nonnegative-definite}, by showing that $(\log L_p(s, \pi \times \tilde\pi'))_{\pi, \pi' \in \mF}$ forms a positive semi-definite family for an arbitrary given prime $p$ (which may be ramified for some $\pi \in \mF$). We assume the setup of \cref{prop:rs-nonnegative-definite} and follow the computations of Rudnick--Sarnak \cite[\S 5]{rudnick1996zeros}; we also point the reader again to the closely-related computations of Brumley in \cite[Appendix]{soundararajan2019weak}.

As in \cite[(5.1)]{rudnick1996zeros}, we can write the local component $\pi_p$ of some $\pi \in \mF$ as a Langlands quotient, the unique irreducible quotient of the induced representation
\[
    \Ind\left(\GL_n, P_\pi; \left(\sigma_{\pi,j}[t_{\pi,j}] \right)_{j = 1}^{J_{\pi}} \right),
\]
where $P_\pi$ is a standard parabolic subgroup of type $(n_{\pi,j})_{j=1}^{J_\pi}$, $\sigma_{\pi,j}$ are unitary \emph{tempered} representations of $\GL_{n_{\pi,j}}$, and $t_{\pi,j}\in\R$ are the Langlands parameters. Here we recall the notation $\sigma[t] := \sigma \otimes |\cdot|^t$ which also applies locally over $\Q_p$; since $\pi$ is unitary, we also have 
$\left\{\sigma_{\pi,j}[t_{\pi,j}]\right\} 
= \left\{\sigma_{\pi,j}[-t_{\pi,j}]\right\}$.
    
Furthermore, as in \cite[(5.3)]{rudnick1996zeros} (see also \cite[Theorem 14.6.5]{goldfeld2011automorphicII}), each tempered $\sigma_{\pi,j}$ is given as an induced representation
\[
    \Ind\left(\GL_{n_{\pi,j}}, P_{\pi,j}; \left(\tau_{\pi,j,k}\right)_{k=1}^{K_{\pi,j}}\right),
\]
where $P_{\pi,j}$ a standard parabolic subgroup of type $(n_{\pi,j,k})_{k=1}^{K_{\pi,j}}$, and $\tau_{\pi,j,k}$ are unitary \emph{square-integrable} (or `\emph{discrete series}') representations of $\GL_{n_{\pi,j,k}}$. In turn, as in \cite[\S 5.2]{rudnick1996zeros} (see also \cite[Theorem 14.6.4]{goldfeld2011automorphicII}), each $\tau_{\pi,j,k}$ is the unique square-integrable constituent of the induced representation
\[
    \Ind\left(\GL_{n_{\pi,j,k}}, P_{\pi,j,k}; (\rho_{\pi,j,k}[\ell - \tfrac{L_{\pi,j,k}+1}{2}])_{\ell=1}^{L_{\pi,j,k}} \right),
\]
for some $L_{\pi,j,k}\mid n_{\pi,j,k}$, and $P_{\pi,j,k}$ is a standard parabolic subgroup of type $(d, \ldots, d)$ with $d= d_{\pi,j,k} := n_{\pi,j,k}/L_{\pi,j,k}$, and $\rho_{\pi,j,k}$ is a unitary \emph{super-cuspidal} representation of $\GL_{d_{\pi,j,k}}$. We can write 
\begin{equation} \label{eq:square-int-split-notation} 
    \tau_{\pi,j,k} = \Delta(L_{\pi,j,k}, \rho_{\pi,j,k}),
\end{equation} 
and the contragredient of such a representation $\Delta(L, \rho)$ is $\tilde{\Delta}(L, \rho) = \Delta(L, \tilde{\rho})$.

Then as in \cite[\S 5.2]{rudnick1996zeros}, the local factor at $p$ of $L(s, \pi)$ splits as a product
\begin{equation} \label{eq:local-rep-split}
\begin{aligned}
    L_p(s, \pi) 
    &= 
    \prod_{j=1}^{J_\pi} L(s+t_{\pi,j}, \sigma_{\pi,j})
    = 
    \prod_{j=1}^{J_\pi} \prod_{k=1}^{K_\pi} 
    L(s+t_{\pi,j}, \tau_{\pi,j,k})
    \\
    &= 
    \prod_{j=1}^{J_\pi} \prod_{k=1}^{K_\pi} 
    L(s+t_{\pi,j}, \Delta(L_{\pi,j,k}, \rho_{\pi,j,k}))
    = 
    \prod_{j=1}^{J_\pi} \prod_{k=1}^{K_\pi} 
    L\left(s+t_{\pi,j}+\tfrac{L_{\pi,j,k}-1}{2},\rho_{\pi,j,k}\right).
\end{aligned}
\end{equation}
For unitary supercuspidal representations, we have
\[
    L(s, \rho_{\pi,j,k}) = 
    \begin{cases}
        \left(1 - p^{-(s+iu)}\right)^{-1}
        & \textnormal{if} \ d_{\pi,j,k}=1,\,\rho = |\cdot|^{iu},\ u \in \R,\\
        1 & \textnormal{else}.
    \end{cases}
\]
Similarly, as in \cite[\S 5.2]{rudnick1996zeros} (see also \cite[(9) and Theorems (9.5), (8.2)]{jacquet1983rankin}), the local factor of the Rankin--Selberg $L$-function of $\pi, \pi' \in \mF$ factors as
\begin{equation} \label{eq:local-rs-split}
\begin{aligned}
    &L_p\left(s, \pi \times \tilde{\pi}'\right)
    \\
    &=
    \prod_{j=1}^{J_\pi}
    \prod_{j'=1}^{J_{\pi'}} L(s+t_{\pi,j}- t_{\pi',j'}, \sigma_{\pi,j} \times \tilde\sigma_{\pi',j'})
    \\
    &=
    \prod_{j=1}^{J_\pi}
    \prod_{k=1}^{K_{\pi,j}}
    \prod_{j'=1}^{J_{\pi'}}
    \prod_{k'=1}^{K_{\pi',j'}}
    L\left(s+t_{\pi,j}-t_{\pi',j'}, \tau_{\pi,j,k}\times \tilde\tau_{\pi',j',k'}\right)
    \\
    &= 
    \prod_{j=1}^{J_\pi}
    \prod_{k=1}^{K_{\pi,j}}
    \prod_{j'=1}^{J_{\pi'}}
    \prod_{k'=1}^{K_{\pi',j'}}
    L\left(s+t_{\pi,j}-t_{\pi',j'}, \Delta(L_{\pi,j,k}, \rho_{\pi,j,k}) \times \Delta(L_{\pi',j',k'}, \tilde\rho_{\pi',j',k'})\right)
    \\
    &=
    \prod_{j=1}^{J_\pi}
    \prod_{k=1}^{K_{\pi,j}}
    \prod_{j'=1}^{J_{\pi'}}
    \prod_{k'=1}^{K_{\pi',j'}}
    \prod_{\ell = 1}^{\min(L_{\pi,j,k}, L_{\pi',j',k'})}
    L\left(s+t_{\pi,j}-t_{\pi',j'} + \frac{L_{\pi,j,k} + L_{\pi',j',k'}}{2} - \ell, \rho_{\pi,j,k}\times \tilde\rho_{\pi',j',k'}\right),
\end{aligned}
\end{equation}
where we expanded the Rankin--Selberg factor $L(s, \Delta(L, \rho) \times \Delta(L', \rho'))$ for square-integrable representations as in \cite[(5.5)]{rudnick1996zeros}. For unitary supercuspidal representations, we have
\begin{equation} \label{eq:rs-supercuspidal}
    L(s, \rho \times \tilde \rho') = 
    \begin{cases}
        \left(1 - p^{-r(s+iu)}\right)^{-1}
        & \textnormal{if}\ \rho \simeq \rho'[iu] \text{ for some } u \in \R,\\
        1 & \textnormal{otherwise},
    \end{cases}
\end{equation}
where in the first case, $r = r(\rho)$ is the order of the cyclic group of characters $|\cdot|^{iv}$ such that $\rho \simeq \rho[iv]$; note that this $r$ only depends on the twist class of $\rho$ among unitary supercuspidals, 
\[
    [\rho] := \{\rho[iv] : v \in \R\},
\]
so we may write $r = r_{[\rho]}$.

In particular, the product in \cref{eq:local-rs-split} will only pick up factors where $\rho_{\pi,j,k}$ and $\tilde{\rho}_{\pi',j',k'}$ are in the same twist class, so it is natural to split it into a (finite) product over all twist classes $[\rho]$ of the unitary supercuspidals that arise in \cref{eq:local-rep-split} for some $\pi \in \mF$. This yields
\begin{equation} \label{eq:split-twist-classes}
    L_p(s, \pi \times \tilde{\pi}') = \prod_{[\rho]} L_{[\rho]}(s, \pi \times \tilde{\pi}'),
\end{equation}
where by \cref{eq:local-rs-split} and \cref{eq:rs-supercuspidal},
\[
\begin{aligned}
    &L_{[\rho]}(s, \pi \times \tilde{\pi}') :=
    \\
    &
    \prod_{\substack{1 \le j \le J_\pi \\ 1 \le k \le K_{\pi,j} \\ \rho_{\pi,j,k} = \rho[iu_{\pi,j,k}]}}
    \,
    \prod_{\substack{1 \le j' \le J_{\pi'} \\ 1 \le k \le K_{\pi',j'} \\ \rho_{\pi',j',k'} = \rho[iu_{\pi',j',k'}]}}
    \prod_{\ell = 1}^{\min(L_{\pi,j,k}, L_{\pi',j',k'})}
    \left(1 - p^{-r_{[\rho]}\left(s+s_{\pi,j,k}-s_{\pi',j',k'} + \frac{L_{\pi,j,k} + L_{\pi',j',k'}}{2} - \ell \right)}\right)^{-1},
\end{aligned}
\]
for $s_{\pi,j,k} := t_{\pi,j} + iu_{\pi,j,k}$; this corresponds to \cite[(5.9)]{rudnick1996zeros}. Note that for each $[\rho]$ and each $\pi$, by unitarity, we have $\{(s_{\pi,j,k}, L_{\pi,j,k}) : \rho_{\pi,j,k} \in [\rho]\} = \{(-\bar s_{\pi,j,k}, L_{\pi,j,k}) : \rho_{\pi,j,k} \in [\rho]\}$. Thus we can rewrite
\[
\begin{aligned}
    &L_{[\rho]}(s, \pi \times \tilde{\pi}') =
    \\
    &
    \prod_{\substack{1 \le j \le J_\pi \\ 1 \le k \le K_{\pi,j} \\ \rho_{\pi,j,k} \in [\rho]}}
    \,
    \prod_{\substack{1 \le j' \le J_{\pi'} \\ 1 \le k \le K_{\pi',j'} \\ \rho_{\pi',j',k'} \in [\rho]}}
    \prod_{\ell = 1}^{\min(L_{\pi,j,k}, L_{\pi',j',k'})}
    \left(1 - p^{-r_{[\rho]}\left(s+s_{\pi,j,k}+ \bar s_{\pi',j',k'} + \frac{L_{\pi,j,k} + L_{\pi',j',k'}}{2} - \ell \right)}\right)^{-1},
\end{aligned}
\]
By \cref{eq:split-twist-classes} and \cref{lem:properties-nonnegative-definite}, it suffices to prove that for each class $[\rho]$, the family $(\log L_{[\rho]}(s, \pi \times \tilde{\pi}'))_{\pi,\pi' \in \mF}$ is positive semi-definite. In fact, we can further split
\[
    L_{[\rho]}(s, \pi \times \tilde{\pi}')
    =
    \prod_{\ell=1}^\infty L_{[\rho],\ell}(s, \pi \times \tilde{\pi}'),
\]
where
\[
\begin{aligned}
    &L_{[\rho],\ell}(s, \pi \times \tilde{\pi}') :=
    \\
    &
    \prod_{\substack{1 \le j \le J_\pi \\ 1 \le k \le K_{\pi,j} \\ \rho_{\pi,j,k} \in [\rho] \\ L_{\pi,j,k} \ge \ell}}
    \,
    \prod_{\substack{1 \le j' \le J_{\pi'} \\ 1 \le k \le K_{\pi',j'} \\ \rho_{\pi',j',k'} \in [\rho] \\ L_{\pi',j',k'} \ge \ell}}
    \left(1 - p^{-r_{[\rho]}\left(s+s_{\pi,j,k}+\bar s_{\pi',j',k'} + \frac{L_{\pi,j,k} + L_{\pi',j',k'}}{2} - \ell \right)}\right)^{-1},
\end{aligned}
\]
and it suffices to prove that $(\log L_{[\rho],\ell}(s, \pi \times \tilde{\pi}'))_{\pi,\pi' \in \mF}$ is positive semi-definite for each $\ell$. One can expand the formal logarithm as
\[
\begin{aligned}
    \log L_{[\rho],\ell}(s, \pi \times \tilde{\pi}')
    &= \sum_{\substack{1 \le j \le J_\pi \\ 1 \le k \le K_{\pi,j} \\ \rho_{\pi,j,k} \in [\rho] \\ L_{\pi,j,k} \ge \ell}}
    \,
    \sum_{\substack{1 \le j' \le J_{\pi'} \\ 1 \le k \le K_{\pi',j'} \\ \rho_{\pi',j',k'} \in [\rho] \\ L_{\pi',j',k'} \ge \ell}}
    -\log \left(1 - p^{-r_{[\rho]}\left(s+s_{\pi,j,k}+\bar s_{\pi',j',k'} + \frac{L_{\pi,j,k} + L_{\pi',j',k'}}{2} - \ell \right)}\right)
    \\
    &= \sum_{\substack{1 \le j \le J_\pi \\ 1 \le k \le K_{\pi,j} \\ \rho_{\pi,j,k} \in [\rho] \\ L_{\pi,j,k} \ge \ell}}
    \,
    \sum_{\substack{1 \le j' \le J_{\pi'} \\ 1 \le k \le K_{\pi',j'} \\ \rho_{\pi',j',k'} \in [\rho] \\ L_{\pi',j',k'} \ge \ell}}
    \sum_{q \ge 1} \frac{1}{q} p^{-qr_{[\rho]}\left(s+s_{\pi,j,k}+\bar s_{\pi',j',k'} + \frac{L_{\pi,j,k} + L_{\pi',j',k'}}{2} - \ell \right)},
\end{aligned}
\]
which can be rearranged to
\[
    \sum_{q \ge 1} \frac{1}{q} p^{-qr_{[\rho]}(s-\ell)} \sum_{\substack{1 \le j \le J_\pi \\ 1 \le k \le K_{\pi,j} \\ \rho_{\pi,j,k} \in [\rho] \\ L_{\pi,j,k} \ge \ell}}
    p^{-qr_{[\rho]}\left(s_{\pi,j,k} + \frac{L_{\pi,j,k}}{2} \right)}
    \bar{ \sum_{\substack{1 \le j' \le J_{\pi'} \\ 1 \le k \le K_{\pi',j'} \\ \rho_{\pi',j',k'} \in [\rho] \\ L_{\pi',j',k'} \ge \ell}}
    p^{-qr_{[\rho]}\left(s_{\pi',j',k'} + \frac{L_{\pi',j',k'}}{2} \right)} }.
\]
But this is a positive-linear combination of terms $w_{k,\pi} \bar w_{k,\pi'}\, p^{-ks}$, and thus positive semi-definite. This completes our proof.

\subsection{Twists by ramified characters} \label{subsec:twist-ram-chars}

Here we complete the proof of \cref{lem:twisted-epsilon-rs-ram}, by establishing \cref{eq:ram-twist-local}. Let $q$ be a prime and $\chi$ be a primitive even Dirichlet character mod $q$. We abuse notation slightly by writing $\chi$ for the local character induced by $\chi$ on $\GL_1(\Q_q)$; we also write $\one$ for the trivial character on either $\GL_1(\A_\Q)$ or $\GL_1(\Q_q)$ depending on the context. The local $\epsilon$-factors will implicitly depend on an everywhere-normalized additive character $\psi = \prod_v \psi_v$ for $\A_\Q/\Q$.

Let $\pi, \pi'$ be cuspidal automorphic representations of $\GL_n(\A_\Q)$ with
\begin{equation} \label{eq:nice-cond-and-char}
    \mf_\pi = \mf_{\pi'} = q,
    \qquad\qquad 
    \omega_\pi = \omega_{\pi'} = \one.
\end{equation}

We must compute the local factors $L_q(s, (\pi \otimes \chi) \times \tilde \pi')$ and $\epsilon_q(s, (\pi \otimes \chi) \times \tilde\pi')$.
To this end, we express the local components $\pi_q$, $\pi_q'$ as the unique irreducible quotients of the induced representations
\[
    \Ind(\GL_n, P; (\sigma_j[t_j])_{j=1}^J),
    \qquad\qquad 
    \Ind(\GL_n, P'; (\sigma'_{j'}[t'_{j'}])_{j'=1}^{J'}),
\]
where $P, P'$ are standard parabolic subgroups of types $(n_j)_{j=1}^J$, $(n'_{j'})_{j'=1}^{J'}$, the unitary representations $\sigma_{j}$, $\sigma'_{j'}$ of $\GL_{n_j}$, $\GL_{n'_{j'}}$ are \emph{tempered}, and $t_j, t'_{j'} \in \R$. 

As in \cref{subsec:pos-def-ram-primes}, by \cite[Theorem 14.6.5]{goldfeld2011automorphicII}, we can further express each $\sigma_j$ and $\sigma_{j'}'$ as the induced representations
\[
    \Ind\left(\GL_{n_j}, P_j; (\tau_{j,k})_{k=1}^{K_j}\right),
    \qquad\qquad
    \Ind\left(\GL_{n_{j'}'}, P'_{j'}; (\tau'_{j',k'})_{k'=1}^{K'_{j'}}\right),
\]
where $P_j, P'_{j'}$ are standard parabolic subgroups of type $(n_{j,k})_{k=1}^{K_j}$, $(n'_{j',k'})_{k'=1}^{K'_{j'}}$, and $\tau_{j,k}$, $\tau_{j',k'}$ are unitary \emph{square-integrable} representations of $\GL_{n_{j,k}}$, $\GL_{n'_{j',k'}}$. The group ranks satisfy
\begin{equation} \label{eq:group-ranks}
    \sum_{j=1}^J n_j = n,
    \qquad\qquad 
    \sum_{k=1}^{K_j} n_{j,k} = n_j,
\end{equation}
and similarly for $n'_{j'}$, $n'_{j',k'}$. The local $L$-factor and $\epsilon$-factor of the twisted Rankin--Selberg pair can then be expressed as (see \cite[(9), (10), and Theorems (3.1), (9.5)]{jacquet1983rankin})
\begin{equation} \label{eq:local-twist-dec}
\begin{aligned}
    L_q(s, (\pi \otimes \chi) \times \tilde\pi')
    &=
    \prod_{j=1}^{J} \prod_{k=1}^{K_j} \prod_{j'=1}^{J'} \prod_{k'=1}^{K'_{j'}}
    L\left(s + t_j - t'_{j'}, (\tau_{j,k} \otimes \chi) \times \tilde \tau'_{j',k'}\right),
    \\
    \epsilon_q(s, (\pi \otimes \chi) \times \tilde\pi')
    &=
    \prod_{j=1}^{J} \prod_{k=1}^{K_j} \prod_{j'=1}^{J'} \prod_{k'=1}^{K'_{j'}}
    \epsilon\left(s + t_j - t'_{j'}, (\tau_{j,k} \otimes \chi) \times \tilde \tau'_{j',k'}\right).
\end{aligned}
\end{equation}

We use the highly-restrictive conditions from \cref{eq:nice-cond-and-char} through the lemma below; the key point is that there is `not enough room' for any nontrivial supercuspidal constituents, by a similar reasoning as in \cite[\S 2]{blomer2025moments}. Here and throughout, for $m \in \Z_+$, we denote by $\St_m$
the Steinberg representation of $\GL_m(\Q_q)$, normalized to be unitary (and square-integrable); with the notation from \cref{eq:square-int-split-notation}, this is
\begin{equation} \label{eq:stm}
    \St_m = \Delta\left(m, \one\right),
\end{equation}
so in particular $\tilde \St_m = \St_m$ and $\St_1 = \one$.

\begin{lemma}[Special structure of local components] \label{lem:special-structure}
There exist indices $j_0, k_0$, $j_0', k_0'$, and real numbers $(u_{j,k})$, $(u'_{j',k'})$, such that
\[
\begin{aligned}
    (n_{j,k}, \tau_{j,k}[-iu_{j,k}]) &=
    \begin{cases}
        (2, \St_2), & (j, k) = (j_0, k_0), \\ 
        (1, \one), & \text{otherwise},
    \end{cases}
    \\
    (n_{j',k'}, \tau'_{j',k'}[-iu'_{j',k'}]) &=
    \begin{cases}
        (2, \St_2), & (j', k') = (j_0', k_0'), \\ 
        (1, \one), & \text{otherwise}.
    \end{cases}
\end{aligned}
\]
\end{lemma}

\begin{proof}
We prove the claim about $(n_{j,k}, \tau_{j,k})$; the other one follows identically. Writing $a(\tau) = \log_q \mf_\tau$ for the local conductor exponents at $q$ (so in particular $a(\chi) = 1$), we have (see, e.g., \cite{corbett2019explicit})
\[
    1 = a(\pi_q) = \sum_{j=1}^J \sum_{k=1}^{K_j} a(\tau_{j,k}). 
\]
This forces $a(\tau_{j_0,k_0}) = 1$ for a single pair of indices $(j_0, k_0)$, and $a(\tau_{j,k}) = 0$ for $(j, k) \neq (j_0, k_0)$. 

Now the unitary square-integrable representations $\tau_{j,k}$ of $\GL_{n_{j,k}}(\Q_q)$ are classified in \cite[Theorem 14.6.4]{goldfeld2011automorphicII}; see also \cite{assing2024observation,corbett2019explicit} and \cite[\S 2]{blomer2025moments}. The only possibility with $a(\tau_{j,k}) = 0$ is an unramified local character $|\cdot|^{iu_{j,k}}$, so we have $(n_{j,k}, \tau_{j,k}[-iu_{j,k}]) = (1, \one)$ when $(j, k) \neq (j_0, k_0)$. There are two possibilities with $a(\tau_{j_0,k_0}) = 1$:
\begin{itemize}
    \item[(1).] $n_{j_0,k_0} = 1$ and $\tau_{j_0,k_0}$ is a ramified character. Since all other square-integrable constituents are unramified, this would force the central character of $\pi_q$ to be ramified.
    \item[(2).] $n_{j_0,k_0} = 2$ and $\tau_{j_0,k_0} = \St_2[iu_{j_0,k_0}]$ is a unitary twist of the Steinberg representation of $\GL_2(\Q_q)$.
\end{itemize}
Since the central character of $\pi_q$ is trivial, the second case holds, and this completes the proof.
\end{proof}

Now given the values $(u_{j,k})$, $(u'_{j',k'})$ from \cref{lem:special-structure}, let us write
\[
    s_{j,k} := t_j + iu_{j,k},
    \qquad\qquad 
    s'_{j',k'} := t'_{j'} + iu'_{j',k'},
\]
so that
\begin{equation} \label{eq:local-twist-imag-shift}
\begin{aligned}
    L\left(s + t_j - t'_{j'}, (\tau_{j,k} \otimes \chi) \times \tilde \tau'_{j',k'}\right)
    &=
    L\left(s + s_j - s'_{j'}, (\tau_{j,k}[-iu_{j,k}] \otimes \chi) \times \tilde \tau'_{j',k'}[iu'_{j',k'}]\right),
    \\
    \epsilon\left(s + t_j - t'_{j'}, (\tau_{j,k} \otimes \chi) \times \tilde \tau'_{j',k'}\right)
    &=
    \epsilon\left(s + s_j - s'_{j'}, (\tau_{j,k}[-iu_{j,k}] \otimes \chi) \times \tilde \tau'_{j',k'}[iu'_{j',k'}]\right).
\end{aligned}
\end{equation}
Since the central characters of $\pi, \pi'$ are trivial, we have
\begin{equation} \label{eq:central-char-condition}
    q^{\sum_{j=1}^J \sum_{k=1}^{K_j} n_{j,k} s_{j,k}} = 1,
    \qquad\qquad 
    q^{\sum_{j'=1}^{J'} \sum_{k'=1}^{K'_{j'}} n_{j',k'} s_{j',k'}'} = 1.
\end{equation}

\begin{lemma}[Local factors for twisted Steinberg representations] \label{lem:local-square-integrable}
For $m, m' \in \{1, 2\}$, one has
\[
    L(s, (\St_m \otimes \chi) \times \St_{m'}) = 1,
    \qquad\qquad
    \epsilon(s, (\St_m \otimes \chi) \times \St_{m'}) = \epsilon_q(s, \chi)^{mm'} = \tau_\chi^{mm'} q^{-mm's}.
\]
\end{lemma}

\begin{proof}
By \cite[(5.5)]{rudnick1996zeros}, the $L$-factor can be expressed as
\[
\begin{aligned}
    L(s, (\St_m \otimes \chi) \times \St_{m'})
    &=
    L(s, \Delta(m, \chi) \times \Delta(m', \one))
    \\
    &=
    \prod_{\ell=1}^{\min(m, m')} L_q(s + \tfrac{m+m'}{2}-\ell, \chi \times \one)
    =
    1,
\end{aligned}
\]
where each $L_q(s, \chi) = 1$ since $\chi$ is ramified at $q$. 

For the $\epsilon$-factor, we find it easier to work on the Weil--Deligne side. By the local Langlands correspondence, we let $r(\tau)$ denote the Weil--Deligne representation associated to a generic, unitary, irreducible, smooth representation $\tau$ of $\GL_n(\Q_q)$. Then we have
\[
    \epsilon(s, (\St_m \otimes \chi) \times \St_{m'})
    =
    \epsilon(s, r(\St_m \otimes \chi) \otimes r(\St_{m'}))
    =
    \epsilon(s, \chi \otimes (r(\St_m) \otimes  r(\St_{m'}))).
\]
Now following Blomer--Comtat \cite[Proof of Proposition 3]{blomer2025moments}, we write $\Sp_m$ for the $m$-dimensional special Weil--Deligne representation (not unitary), and compute
\begin{gather*}
    r(\St_1) = r(\one) = \one, \qquad\qquad r(\St_2) = |\cdot|^{-1/2} \otimes \Sp_2,
    \qquad\qquad 
    r(\St_3) = |\cdot|^{-1} \otimes \Sp_3,
    \\
    r(\St_2) \otimes r(\St_2) = (|\cdot|^{-1/2} \otimes \Sp_2) \otimes (|\cdot|^{-1/2} \otimes \Sp_2) = \one \oplus (|\cdot|^{-1} \otimes \Sp_3)
    =
    \one \oplus r(\St_3).
\end{gather*}
Moreover, by \cite[Proposition 3.2]{assing2024observation}, for $m \in \Z_+$ we have
\[
    \epsilon(s, \chi \otimes r(\St_m)) =
    \epsilon_q(s, \chi)^m.
\]
Using that $\epsilon(s, \tau_1 \oplus \tau_2) = \epsilon(s, \tau_1) \epsilon(s, \tau_2)$ for Weil--Deligne representations $\tau_1, \tau_2$, it follows that
\[
    \epsilon(s, \chi \otimes (r(\St_m) \otimes  r(\St_{m'})))
    =
    \begin{cases}
    \epsilon_q(s, \chi),    & m = m' = 1, \\ 
    \epsilon(s, \chi \otimes r(\St_2)) = \epsilon_q(s, \chi)^2,    & \{m, m'\} = \{1, 2\}, \\ 
    \epsilon_q(s, \chi) \epsilon(s, \chi \otimes r(\St_3)) = \epsilon_q(s, \chi)^4, & m = m' = 2,
    \end{cases}
\]
which completes our computation.
\end{proof}

Now putting together \cref{eq:local-twist-dec,eq:local-twist-imag-shift,lem:special-structure,lem:local-square-integrable}, we find that
\[
    L_q(s, (\pi \otimes \chi) \times \tilde\pi')
    =
    \prod_{j=1}^{J} \prod_{k=1}^{K_j} \prod_{j'=1}^{J'} \prod_{k'=1}^{K'_{j'}}
    1
    =
    1
    ,
\]
and, by also employing \cref{eq:central-char-condition,eq:group-ranks},
\[
\begin{aligned}
    \epsilon_q(s, (\pi \otimes \chi) \times \tilde\pi')
    &=
    \prod_{j=1}^{J} \prod_{k=1}^{K_j} \prod_{j'=1}^{J'} \prod_{k'=1}^{K'_{j'}}
    \tau_\chi^{n_{j,k} n'_{j',k'}} q^{-n_{j,k} n'_{j',k'} (s + s_{j,k} - s'_{j',k'})}
    \\
    &=
    \tau_\chi^{n^2} q^{-n^2 s} q^{- n\sum_{j=1}^J\sum_{k=1}^{K_j} n_{j,k} s_{j,k} + n \sum_{j'=1}^{J'} \sum_{k'=1}^{K'_{j'}} n'_{j',k'} s'_{j',k'}}
    =
    \tau_\chi^{n^2} q^{-n^2 s} 
    ,
\end{aligned}
\]
as required in \cref{eq:ram-twist-local}.

\bibliographystyle{plain}
\bibliography{main}

\end{document}

%% file: preamble.tex
\usepackage[english]{babel}
\usepackage[utf8]{inputenc}
\usepackage[T1]{fontenc}

\usepackage[margin=1in, letterpaper]{geometry}
\usepackage{amssymb,amsmath,amsfonts,amsthm}
\usepackage{bm,bbm,mathrsfs}
\usepackage[dvipsnames]{xcolor}
\usepackage{url}
\usepackage{enumitem}
\usepackage{mathtools}
\usepackage{placeins}
\usepackage{extarrows}
\usepackage{subfig}
\usepackage[font=small,labelfont=bf]{caption}
\usepackage{floatrow}
\usepackage{verbatim}
\usepackage{wrapfig}
\usepackage{tikz}
\usetikzlibrary{arrows,arrows.meta,cd,decorations.pathmorphing,backgrounds,positioning,fit,matrix}
\usepackage{xcolor,colortbl,graphics,graphicx}
\usepackage[colorlinks=true, allcolors=Blue, pdfencoding=auto]{hyperref}
\usepackage{aliascnt}
\usepackage[noabbrev,capitalize]{cleveref}
\crefname{equation}{}{}

\theoremstyle{plain}
\newtheorem{theorem}{Theorem}[section]
\newaliascnt{lemma}{theorem}
\newtheorem{lemma}[lemma]{Lemma}
\aliascntresetthe{lemma}
\crefname{lemma}{Lemma}{Lemmas} 
\newaliascnt{proposition}{theorem}
\newtheorem{proposition}[proposition]{Proposition}
\aliascntresetthe{proposition}
\crefname{proposition}{Proposition}{Propositions}
\newaliascnt{corollary}{theorem}
\newtheorem{corollary}[corollary]{Corollary}
\aliascntresetthe{corollary}
\crefname{corollary}{Corollary}{Corollaries}
\newaliascnt{problem}{theorem}
\newtheorem{problem}[problem]{Problem}
\aliascntresetthe{problem}
\crefname{problem}{Problem}{Problems}
\newaliascnt{example}{theorem}

\aliascntresetthe{example}
\crefname{example}{Example}{Examples}
\newaliascnt{question}{theorem}

\aliascntresetthe{question}
\crefname{question}{Question}{Questions}
\newaliascnt{conjecture}{theorem}

\aliascntresetthe{conjecture}
\crefname{conjecture}{Conjecture}{Conjectures}
\newaliascnt{assumption}{theorem}

\aliascntresetthe{assumption}
\crefname{assumption}{Assumption}{Assumptions}

\theoremstyle{definition}
\newaliascnt{definition}{theorem}
\newtheorem{definition}[definition]{Definition}
\aliascntresetthe{definition}
\crefname{definition}{Definition}{Definitions}
\newaliascnt{notation}{theorem}

\aliascntresetthe{notation}
\crefname{notation}{Notation}{Notations}
\newaliascnt{remark}{theorem}
\newtheorem{remark}[remark]{Remark}
\aliascntresetthe{remark}
\crefname{remark}{Remark}{Remarks}

\crefname{equation}{}{}
\numberwithin{equation}{section}
\allowdisplaybreaks

\newcommand{\St}{\textnormal{St}}
\newcommand{\Sp}{\textnormal{Sp}}
\newcommand{\Z}{\mathbb{Z}}
\newcommand{\R}{\mathbb{R}}
\newcommand{\Q}{\mathbb{Q}}
\newcommand{\C}{\mathbb{C}}

\newcommand{\A}{\mathbb{A}}

\newcommand{\mB}{\mathcal{B}}

\newcommand{\mF}{\mathcal{F}}

\newcommand{\mI}{\mathcal{I}}

\newcommand{\mS}{\mathcal{S}}

\newcommand{\mf}{\mathfrak{f}}
\newcommand{\mfC}{\mathfrak{C}}
\newcommand{\Sym}{\textnormal{Sym}}
\newcommand{\Arg}{\textnormal{Arg}}
\newcommand{\one}{\mathbbm{1}}

\renewcommand{\Re}{\textnormal{Re}}
\renewcommand{\Im}{\textnormal{Im}}
\newcommand{\GL}{\textnormal{GL}}
\newcommand{\SL}{\textnormal{SL}}

\newcommand{\eps}{\varepsilon}

\renewcommand{\bar}{\overline}

\renewcommand{\tilde}{\widetilde}

\renewcommand{\mod}[1]{\ \textnormal{mod } #1}
\renewcommand{\pmod}[1]{\ (\textnormal{mod } #1)}

\newcommand{\Ind}{\textnormal{Ind}}

\definecolor{myBlue}{rgb}{0, 0, 0.6}